\documentclass{amsart}[12pt,letter]

\usepackage{mathtools} 

\usepackage{color,amssymb,enumerate}

\usepackage{mathtools}

\usepackage{tikz}
\usetikzlibrary{arrows}

\usepackage[all,cmtip]{xy}

\usepackage{hyperref}
\usepackage[alphabetic]{amsrefs}

 \usepackage{nicematrix}

\newtheorem{theorem}{Theorem}[section]
\newtheorem{corollary}[theorem]{Corollary}
\newtheorem{lemma}[theorem]{Lemma}
\newtheorem{proposition}[theorem]{Proposition}

\theoremstyle{definition}
\newtheorem{definition}[theorem]{Definition}
\newtheorem{definition*}{Definition}
\theoremstyle{remark}
\newtheorem{remark}[theorem]{Remark}
\numberwithin{equation}{section}

\theoremstyle{remark}
\newtheorem{claim*}[]{Claim}

\newlength{\oldfb}
\newcommand{\boxx}[1]
{\setlength{\oldfb}{\fboxrule}\setlength{\fboxrule}{2pt}\framebox{\parbox{\dimexpr\linewidth-2\fboxsep-2\fboxrule}{#1}}\setlength{\fboxrule}{\oldfb}}

\newcommand{\out}[1]
{#1}

\newcommand{\R}{\mathbb{R}}
\newcommand{\C}{\mathbb{C}}

\newcommand{\Z}{\mathbb{Z}}

\newcommand{\K}{\mathbb{K}}

\newcommand{\End}{\operatorname{End}}

\newcommand{\Y}{\mathcal{Y}}

\newcommand{\A}{\mathcal{A}}
\newcommand{\W}{\mathcal{W}}
\newcommand{\M}{\mathcal{M}}

\newcommand{\Hom}{{\operatorname{Hom}}}

\DeclareMathOperator{\Span}{Span}             
\DeclareMathOperator{\crit}{Crit}             
\DeclareMathOperator{\hol}{hol}             
\DeclareMathOperator{\mon}{mon}             
\DeclareMathOperator{\Ext}{Ext}             
\DeclareMathOperator{\mmod}{mod}             
\DeclareMathOperator{\val}{val}             
\DeclareMathOperator{\ind}{ind}             
\DeclareMathOperator{\Cone}{Cone}             
\DeclareMathOperator{\Tw}{Tw}             
\DeclareMathOperator{\odd}{odd}             
\DeclareMathOperator{\even}{even}             
\DeclareMathOperator{\op}{op}             

\newcommand{\F}{\mathcal{F}}

\begin{document}

\title{Monotone Lagrangians in cotangent bundles of spheres}
\date{\today}
\author{Mohammed Abouzaid}
\address{Department of Mathematics, Columbia University, 2990 Broadway, MC 4406, New York, NY 10027, USA and Department of Mathematics, Stanford University, Building 380, Stanford, California 94305, USA}
\email{abouzaid@math.columbia.edu, abouzaid@stanford.edu}
\author{Lu\'is Diogo}
\address{Department of Mathematics, Uppsala University, Box 480, 751 06 Uppsala, Sweden}
\email{luis.diogo@math.uu.se}

\maketitle

We study the compact monotone Fukaya category of $T^*S^n$, for $n\geq 2$, and show that it is split-generated by two classes of objects: the zero-section $S^n$ (equipped with suitable bounding cochains) and a 1-parameter family of monotone Lagrangian tori $(S^1\times S^{n-1})_\tau$, with monotonicity constants $\tau>0$ 
(equipped with rank 1 unitary local systems). As a consequence, any closed orientable spin monotone Lagrangian (possibly equipped with auxiliary data) with non-trivial Floer cohomology is non-displaceable from either $S^n$ or one of the $(S^1\times S^{n-1})_\tau$. In the case of $T^*S^3$, the monotone Lagrangians $(S^1\times S^2)_\tau$ can be replaced by a family of  monotone tori $T^3_\tau$.

\tableofcontents

\section{Introduction}

An embedded Lagrangian $L$ in a cotangent bundle $(T^*Q,d(pdq))$, is {\em exact} if $pdq|_L = df$ for some function $f: L \to \R$. Arnold's nearby Lagrangian conjecture predicts that if $Q$ and $L$ are closed, then $L$ is Hamiltonian-isotopic to the zero-section $Q \subset T^*Q$. This result is currently known to hold only for a limited list of examples, including $Q = S^2$ \cite{Hind} and $T^2$ \cite{DGI}. 
The work of many authors has also led to a proof that the composition $L \to T^*Q \to Q$ (where the first map is the embedding and the second is projection to the zero-section) is a simple homotopy equivalence \cite{AbouzaidKraghSimple}. 

Very little is known if one drops the requirement of $L$ being exact. We will consider the case of $L$ {\em monotone}, by which we mean that there is a constant $\tau \geq 0$ such that, for every smooth map $u:(D^2 , \partial D^2) \to (T^*Q,L)$,
$$
\int_{D^2} u^*\omega = \tau \cdot \mu(u)
$$
where $\mu(u)$ is the {Maslov index} of $u$. 
Note that we allow the case $\tau=0$, which happens, for instance when $L$ is exact. For some results about monotone Lagrangians in cotangent bundles, see for instance \cite{GadbledCotangent}. 

\begin{remark}
Under suitable conditions, $\tau = 0$ implies that $L$ is exact. This is the case, for example, if the following conditions are both satisfied:
\begin{enumerate}
\item \label{cond:Hurewicz} the map $h^*\colon \Hom(H_2(T^*Q,L),\R) \to \Hom(\pi_2(T^*Q,L),\R)$ is injective ($h^*$ is the dual of the relative Hurewicz homomorphism, see for example \cite{DavisKirk}*{Section 6.17});
\item \label{cond:H1} the map $H^1(T^*Q;\R) \to H^1(L;\R)$ induced by inclusion is trivial.
\end{enumerate}
Recall that there is an isomorphism $U\colon H^2(T^*Q,L;\R) \to  \Hom(H_2(T^*Q,L),\R)$ given by the universal coefficients theorem. Note that if \eqref{cond:Hurewicz} holds, then $h^*\circ U \colon H^2(T^*Q,L;\R) \to \Hom(\pi_2(T^*Q,L),\R)$ is also injective. This is useful when showing that $L$ is exact. 

Condition \eqref{cond:H1} above holds if $Q=S^n$ with $n\geq 2$. The fact that $\R$ is an injective $\Z$-module implies that condition \eqref{cond:Hurewicz} holds if $Q=S^n$ with $n\geq 1$, using the five-lemma (see for example \cite{Hatcher}*{page 129}).
Hence, if $Q=S^n$ with $n\geq 2$, a monotone Lagrangian $L\subset T^*S^n$ is exact iff $\tau=0$. 
\end{remark}

The focus of this paper is on closed monotone Lagrangians in cotangent bundles of spheres $S^n$ (with the standard smooth structure), from the point of view of Floer theory, more specifically using {\em wrapped Floer cohomology}. Unless otherwise specified, we will always assume that $n\geq 2$. Given closed Lagrangians $L,L'$ (possibly equipped with additional data like bounding cochains or local systems) in a symplectic manifold, one can sometimes define their Floer cohomology $HF^*(L,L')$, which is invariant under Hamiltonian perturbations of either $L$ or $L'$. If $HF^*(L,L') \neq 0$, then $L$ is not Hamiltonian-displaceable from $L'$ (which means that $\varphi(L)\cap L' \neq \emptyset$ for every Hamiltonian diffeomorphism $\varphi$) \cite{FloerLagrangian}. Unless we say otherwise, we will take Floer cohomology with coefficients in the Novikov field over $\C$, which is denoted by $\K$ and defined in Section \ref{SS:Coefficients}. 

There is a 1-parameter family of disjoint monotone Lagrangians $(S^1\times S^{n-1})_\tau \subset T^*S^n$, of different monotonicity constants $\tau>0$, whose construction will be reviewed below. These Lagrangians can be equipped with local systems such that their Floer cohomologies are non-trivial. In $T^*S^3$, the same holds for a 1-parameter family of disjoint monotone Lagrangian tori $T^3_\tau$, see \cite{ChanPomerleanoUeda1}. We will review the construction of these tori below as well. We will prove the following result. 

\begin{theorem} \label{T:non-displ}
Take $n\geq 2$ and let $L\subset T^*S^n$ be a closed orientable spin monotone Lagrangian
with a unitary local system of rank 1 for which $HF^*(L,L;\K) \neq 0$. Then, either $HF^*(L,S^n;\K) \neq 0$ (where the zero-section $S^n$ is equipped with a suitable bounding cochain) or there is a $\tau>0$ for which $HF^*(L,(S^1\times S^{n-1})_\tau;\K) \neq 0$ (where $(S^1\times S^{n-1})_\tau$ is equipped with a suitable unitary local system of rank 1). In particular, $L$ is not Hamiltonian-displaceable from either $S^n$ or from $(S^1\times S^{n-1})_\tau$, for some $\tau >0$. 

Furthermore, in $T^*S^3$ we can replace the $(S^1\times S^{2})_\tau$ with the tori $T^3_\tau$. 
\end{theorem}

Our work towards the proof of Theorem \ref{T:non-displ} will also imply the following.

\begin{theorem} \label{T:S1xS2 and T3}
Let $\tau, \tau' > 0$. Then $\tau = \tau'$ iff the Lagrangians $(S^1 \times S^2)_\tau$ and $T^3_{\tau'}$ can be equipped with unitary local systems of rank 1 with respect to which $HF^*((S^1 \times S^2)_\tau,T^3_{\tau'};\K) \neq 0$. In particular, $(S^1 \times S^2)_\tau$ is not Hamiltonian-displaceable from $T^3_\tau$. 
\end{theorem}

We now describe the structure of the proof of Theorem \ref{T:non-displ}. The Lagrangians $L$ in the statement give objects in a {monotone wrapped Fukaya} category of $T^*S^n$, which also includes a cotangent fiber $F = T^*_q S^n$ (for some $q\in S^n$). This is an $A_\infty$-category (with only a $\Z/2\Z$-grading, since we allow monotone Lagrangians), which we denote temporarily by $\W$ (and will refer to it as $\W^{\Z/2\Z}_{\mon}(Y;\K)$ 
in Section \ref{S:W}). 

The category $\W$ is split-generated by the cotangent fiber $F$, as can be proven by adapting a result in \cite{AbouzaidCotangentFiber} (which in its original form was for the wrapped Fukaya category of exact Lagrangians). Let us consider some algebraic consequences of this split-generation result. Let $A_\K := HW^*(F,F;\K)$ be the wrapped Floer cohomology algebra of $F$. The graded algebra $A_\K$ is isomorphic to $H_{-*}(\Omega_q S^n;\K)$, where $q\in S^n$ is a basepoint and $\Omega_q$ denotes the based loop space, see \cite{AbouzaidBasedLoops}. Hence, since $n\geq 2$, $A_\K$ is isomorphic to a polynomial algebra $\K[u]$, where $\deg(u) = 1-n$. There is a {\em Yoneda functor}
\begin{align*}
Y : \W &\to \mmod(A_\K) \\
L & \mapsto HF^*(F,L;\K)
\end{align*}
where $\mmod(A_\K)$ is the category of $\Z/2\Z$-graded right $A_\K$-modules., with the morphism space between two objects $M,M'$ in $\mmod(A_\K)$ being $\Ext_{A_\K}^*(M,M')$.

The split-generation result mentioned above, together with formality results for $A_\infty$-modules over $A_\K$ that we prove in Section \ref{S:Formality algebra}, imply that $Y$ is a cohomologically full and faithful functor, which means that it induces isomorphisms on cohomology
$$
HW^*(L,L';\K) \cong \Ext_{A_\K}^*(Y(L),Y(L'))
$$   
for any pair of objects $L,L'$. See Proposition \ref{C:Yoneda ff} below for more details. Take the subcategory $\F \subset \W$ (denoted as $\F^{\Z/2\Z}_{\mon}(Y;\K)$ in Section \ref{S:W}) that does not include the object $F$, but only compact monotone Lagrangians. Given an object $L$ of this subcategory (where we suppress the additional data of local systems or bounding cochains), $HW^*(F,L;\K) = HF^*(F,L;\K)$ is a finite dimensional $\K$-vector space, so $Y$ restricts to a cohomologically full and faithful embedding 
\begin{equation} \label{def:Y_c}
Y_c : \F \to \mmod_{pr}(A_\K)
\end{equation}
into the subcategory $\mmod_{pr}(A_\K) \subset \mmod(A_\K)$ of proper $A_\K$-modules $M$ (those which are finite dimensional over $A_\K$).  The approach of this paper will be to study the category $\F$ by analyzing the algebraic category $\mmod_{pr}(A_\K)$. 

Corollaries \ref{S generate} and \ref{S tilde generate} below give generators for $\mmod_{pr}(A_\K)$, and imply the following result (which will appear below as Corollaries \ref{generate Fuk Sn} and \ref{generate Fuk}). 

\begin{theorem} \label{generate F}
Given $n\geq 2$, the functor $Y_c$ in \eqref{def:Y_c}, when extended to the split-closure of $\mathcal F$ (which is the monotone Fukaya category of $T^*S^n$), is a quasi-equivalence of categories. 
The category $\F$ is split-generated by the uncountable collection of objects consisting of $S^n$ (equipped with suitable bounding cochains) and the $(S^1\times S^{n-1})_\tau$ (equipped with unitary local systems of rank 1). In the case of $T^*S^3$, we can replace the $(S^1\times S^{2})_\tau$ with the tori $T^3_\tau$. 
\end{theorem}

\begin{proof}[Proof of Theorem \ref{T:non-displ}] 
Given $L$ such that $HF^*(L,L;\K) \neq 0$, $L$ is a non-trivial object in $\F$. Theorem \ref{generate F} then implies that $HF^*(L,L';\K) \neq 0$, where $L'$ is one of the split-generators. 
\end{proof}

\begin{remark}[Relation to mirror symmetry]

As mentioned, the tori $T^3_\tau$ were studied in \cite{ChanPomerleanoUeda1}. They are fibers of an SYZ fibration in the complement of an anticanonical divisor $H$ in $T^*S^3$ ($H$ is anticanonical in the sense that the Lagrangian tori in the SYZ fibration have vanishing Maslov class in the complement of $H$). In this setting, the authors compute the disk potentials associated to SYZ-fibers by studying wall-crossing for pseudoholomorphic disks. This information is used to construct a Landau--Ginzburg model that is mirror to $T^*S^3$. The critical locus of the Landau--Ginzburg potential is an affine line. If the mirror is constructed over the Novikov field, then the points in this critical line with negative valuation correspond to (split summands of) the monotone Lagrangians $T^3_\tau$, equipped with suitable unitary local systems of rank 1. The points with non-negative valuation correspond to bounding cochains on the zero section $S^3$. 
\end{remark}

\begin{remark}[Relation to abstract flux] \label{rmk:flux}

The monotone Lagrangians $(S^1\times S^{n-1})_\tau$ can be obtained geometrically as follows. Let $f:S^n \to \R$ be a Morse function with exactly two critical points. The graph of $df$ intersects the zero section of $T^*S^n$ transversely in the two critical points, and one can perform surgery on this transverse intersection
to produce the family $(S^1\times S^{n-1})_\tau$. Similarly, the tori $T^3_\tau$ can be obtained by taking a Morse--Bott function $g:S^3 \to \R$ whose critical locus is a Hopf link, and performing surgery in $T^*S^3$ on the clean intersection of the zero section and the graph of $dg$. 

Recall that given a compact manifold $Q$ and a class $\alpha \in H^1(Q;\R)$, one can take the {\em flux deformation} of the zero-section of $T^*Q$ in the direction of $\alpha$, by flowing $Q$ along a symplectic vector field $X$ such that $[\omega(.,X)] = i^*\alpha$ (where $i:Q\to T^*Q$ is the inclusion). Using the Weinstein tubular neighborhood theorem, one can similarly deform a compact Lagrangian $L$ in a symplectic manifold $(M,\omega)$ along a class $\alpha \in H^1(L;\R)$. 
Motivated by \cite{SeidelAbstractFlux}, one can think of the family of Lagrangians $(S^1\times S^{n-1})_\tau$ (respectively, $T^3_\tau$) as an {\em abstract flux deformation} of two copies of the zero section $S^n$ (respectively, $S^3$) in the direction of a class $\beta\in H^n(S^n;\R)$ (respectively, $H^3(S^3;\R)$), if $n$ is odd. The case of $n$ even is more subtle, as we will see.  
\end{remark}

This paper is organized as follows. In Section \ref{S:construct Lagrangians}, we present the construction of the monotone Lagrangians $(S^1 \times S^{n-1})_\tau$ in $T^*S^n$ and $T^3_\tau$ in $T^*S^3$. In Section \ref{S:Fukaya categories}, we recall the definitions of several versions of Fukaya categories of $T^*S^n$, including a monotone wrapped Fukaya category where Lagrangians are allowed to intersect cleanly. In Section \ref{S:HF computations}, we perform several Floer cohomology computations, with a view towards proving Theorem \ref{generate F}. The remaining sections have a more algebraic nature, and are about $A_\infty$-algebras and $A_\infty$-modules. In Section \ref{S:Formality algebra}, we establish formality results for a category of modules associated to a cotangent fiber in $T^*S^n$. In Section \ref{S:Generation modules}, we obtain generators for that category of modules.

\subsection*{Acknowledgements} 
The first named author was supported by the Simons Foundation through its ``Homological Mirror Symmetry'' Collaboration grant SIMONS 385571, and by NSF grants DMS-1609148, and DMS-1564172.

The second named author thanks Yank{\i} Lekili, Maksim Maydanskiy and Daniel Pomerleano for helpful conversations. 
The authors thank an anonymous referee for helpful comments.

\section{Monotone Lagrangians in $T^*S^n$}

\label{S:construct Lagrangians}

\subsection{Lagrangians in $T^*S^n$} \label{SS:Lagrs in T*Sn}
Recall that $T^*S^n$ is symplectomorphic to the complex affine quadric 
$$
X_n = \{(z_0, \ldots, z_{n}) \in \C^{n+1} \,|\, z_0^2 + \ldots + z_{n}^2  = 1 \},
$$
equipped with the K\"ahler form $\omega$ obtained from the restriction of $\frac{i}{2}\sum_{j=0}^{n} d z_j \wedge d \overline{z_j}$ on $\C^{n+1}$ \cite{McDuffSalamonIntro}*{Exercise 6.20}. Assume that $n\geq 2$. The projection to the first coordinate defines a Lefschetz fibration 
\begin{align*}
\pi_n\colon X_n &\to \C \\
(z_0, \ldots, z_{n}) &\mapsto z_0
\end{align*}
with critical values $\pm 1$. For every regular value $p\neq \pm 1$, the fiber $\pi_n^{-1}(p)$ is symplectomorphic to $T^*S^{n-1}$, and contains the Lagrangian sphere 
$$
V_{p}:= \{(p,\sqrt{1-p^2} \, x_1, \ldots, \sqrt{1-p^2} \, x_{n}) \in X_n \, | \, (x_1, \ldots, x_{n}) \in S^{n-1}\},
$$
where $S^{n-1}\subset \R^{n}$ is the unit sphere and $\sqrt{1-p^2}$ is one of the two square roots of $1-p^2$. Write also $V_{\pm1} = \{(\pm1,0,\ldots,0)\}$.

\begin{figure}
  \begin{center}
    \def\svgwidth{0.5\textwidth}
    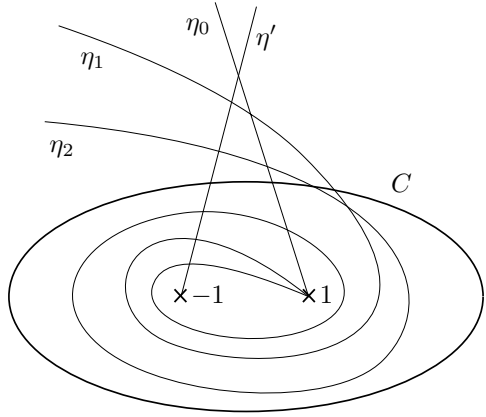
  \end{center}
  \caption{Some curves in $\C\setminus \{\pm 1\}$}
  \label{F_w_fig}
\end{figure}

We will be interested in the following types of Lagrangians that project to curves under $\pi_n$. See Figure \ref{F_w_fig} for relevant examples of such curves. 

\begin{definition}\label{D:Lagr F}
 Given a curve $C\subset \C\setminus \{-1, 1\}$ that is the image of an embedding of $S^1$, let 
 $$
 L_C := \bigcup_{z\in C} V_{z}.
 $$

 Given an embedding $\eta: [0,\infty) \to \C$ such that 
 \begin{itemize}
  \item $\eta(0) \in \{-1,1\}$,
  \item $\eta\big((0,\infty)\big) \subset \C\setminus \{-1,1\}$ and
  \item $\eta(t) = at+b$ for some $a\in \C^*$, $b\in \C$ and $t$ large enough, let
 \end{itemize}
\begin{equation*}
F_\eta :=  \bigcup_{t\geq 0} V_{\eta(t)}.
\end{equation*}
\end{definition}

\begin{lemma} \label{L:L_C and F_eta}
The subsets $L_C$ and $F_\eta$ of  $X_n$ in Definition \ref{D:Lagr F} are Lagrangian submanifolds. If $C$ encloses both points $\pm 1$, then $L_C$ is diffeomorphic to $S^1\times S^{n-1}$, while
$F_\eta$ is Hamiltonian-isotopic to a cotangent fiber in $T^*S^n$. 
\end{lemma}
\begin{proof}
The $L_C$ and $F_\eta$ are Lagrangians because parallel transport with respect to the connection induced by the symplectic fibration $\pi_n$ preserves the spheres $V_p$ (they are vanishing cycles for arbitrary vanishing paths in the base), see \cite{SeidelBook}*{Lemma 16.3}. 

Since there are only two types of $S^{n-1}$-bundles over $S^1$, and the closed curve $C$ encircles two critical values which have the same monodromy (a Dehn twist), it follows that $L_C$ is the trivial bundle.  

We now consider the Lagrangians $F_\eta$. Take $\eta_\pm$ such that $\eta_\pm(t)=\pm(t+1)$ for all $t\geq 0$. Then, $F_{\eta_\pm}$ is mapped to $T_{\pm 1}^*S^n$ by the symplectomorphism $X_n\to T^*S^n$ in \cite{McDuffSalamonIntro}*{Exercise 6.20}. For any other $\eta$, one can construct an isotopy to one of the $\eta_\pm$ that lifts to a Hamiltonian isotopy by applying Moser's trick. 
\end{proof}

\begin{remark}
The Floer cohomology of the Lagrangian submanifolds $L_C\cong S^1\times S^{n-1}$ in $T^*S^n$ in the previous lemma was studied in \cite{AlbersFrauenfelderTorus}. These are a particular case of the {\em generalized Polterovich Lagrangians} in \cite{OakleyUsher}.
\end{remark}

\begin{remark}
In this Lefschetz fibration description $\pi_n : X_n\to \C$ of $T^*S^n$, the zero section $S^n$ is the Lagrangian lift of the interval $[-1,1] \subset \C$.
\end{remark}

Let us continue with our study of the Lagrangians $L_C$, where $C$ encloses $\{\pm1\}$. Much of what follows in this section is an adaptation of results in \cite{LekiliMaydanskiy}*{Section 2.2}. The homology long exact sequence of the pair $(T^*S^n,L_C)$ implies that 
$$
H_2(T^*S^n,L;\Z) \cong H_2(T^*S^n;\Z) \oplus H_1(L_C;\Z)
$$
if $n\geq 2$. The group $H_2(T^*S^n;\Z)$ vanishes unless $n=2$, in which case both $\omega$ and $c_1(T^*S^2)$ vanish on $H_2(T^*S^2;\Z)$. For $n\geq 3$, the group $H_1(L_C;\Z)$ has rank 1 and $H_2(T^*S^n,L_C;\Z)\cong \Z$ is generated by a class $\beta$ such that $\pi_n\circ \beta$ covers $C$ once. For $n=2$, $H_1(L_C;\Z)$ has rank 2 and we can pick $\alpha, \beta\in H_2(T^*S^2,L_C;\Z)$ such that their boundaries give a basis for $H_1(L_C;\Z)\cong \Z^2$, with the following properties: $\alpha$ is a Lefschetz thimble for some vanishing cycle $V_p$ (hence it has vanishing Maslov index and symplectic area), while the boundary of $\pi_2\circ \beta$ covers $C$ once. We will now study the $\omega$-area and Maslov index of the disks $\beta$.   

We need some auxiliary notation. 
Denote by $\sigma_{std}:= \frac{i}{2} dz\wedge d\overline {z} = r dr \wedge d\theta$ the standard area form in $\C$. Define, on the set $\C \setminus \{\pm 1\}$ of regular values of the Lefschetz fibration $\pi_n$, the 2-form
$$
\sigma := \frac{i}{2} dz_0\wedge d\overline {z_0} + f^*\sigma_{std},
$$
where $f : \C \setminus \{\pm 1\} \to \C\setminus \{0\}$ is given by $f(z) = \frac{{1-z^2}}{\sqrt{2|1-z^2|}}$. The function $f$ can be thought of as the composition of the two maps
\begin{align*}
\C\setminus\{\pm 1\} &\to \C\setminus\{0\} & \C\setminus \{0\} &\to \C\setminus\{0\} \\
z &\mapsto 1-z^2   & r e^{i\theta} &\mapsto \sqrt{\frac{r}{2}} e^{i\theta}
\end{align*} 
The first map is holomorphic and the second is smooth and orientation-preserving, so $\sigma$ defines a positive measure on $\C\setminus \{\pm 1\}$. It extends to all of $\C$, as a measure that is absolutely continuous with respect to the Lebesgue measure. 

\begin{lemma} \label{L:sigma area}
Given a disk $\beta: (D^2,\partial D^2)\to (X_n,L_C)$ such that $\pi_n\circ \beta$ covers $C$ once, we have
$$
\int_\beta \omega = \int_{\pi_n (\beta)} \sigma.
$$
\end{lemma}
\begin{proof}
Take $\beta$ as in the lemma. We can assume the boundary of $\beta$ to be given by $c(t) = \left(\gamma(t), \sqrt{1-\gamma(t)^2}\, s(t)\right)$, where $\gamma:[0,1]\to \C\setminus \{\pm 1\}$ is a degree 1 parametrization of $C$ and $s(t) = (s_1(t),\ldots,s_{n+1}(t))\in S^{n-1}$. Here, $\sqrt{.}$ is the analytic continuation of a branch of the square root along the path $1-\gamma^2$. Write $g(t):= \sqrt{1-\gamma(t)^2}$. We have 
\begin{equation} \label{difference}
\int_\beta \omega - \frac{i}{2} dz_0\wedge d\overline {z_0} = \int_c \sum_{j=1}^{n} \frac{i}{4} (z_j \, d\overline{z_j} - \overline{z_j} \, d{z_j}) = \frac{i}{4} \int_0^1 g \overline g' - \overline g g' dt,
\end{equation}
using on the first identity Stokes' theorem and the fact that $\frac{i}{4} (z_j \, d\overline{z_j} - \overline{z_j} \, d{z_j})$ is a primitive of $\frac{i}{2} dz_j\wedge d\overline{z_j}$. 
A calculation shows that the right side of \eqref{difference} can be written as
$$
\int_0^1 g^*\left(\frac{1}{2}r^2 d\theta\right) = \int_{C}f^*\left(\frac{1}{2} r^2 d\theta\right) ,
$$
where $f$ is the function defined before the lemma. Identifying $C$ with the boundary of $\pi_n(\beta)$ and using Stokes' theorem, the integral on the right equals
$$
\int_{\pi_n(\beta)} f^*\sigma_{std},
$$
which finishes the proof. 
\end{proof}

\begin{remark}
The previous argument also goes through if $C$ is a piecewise smooth curve. This will be helpful in Section \ref{S:HF computations}, when computing operations $\mu^k$ involving several Lagrangians that fiber over paths in $\C$. 
\end{remark}

\begin{corollary}
Suppose that the simple curves $C$ and $C'$ in $\C\setminus \{-1, 1\}$ both enclose $\{-1,1\}$. Then, they bound the same $\sigma$-area if and only if $L_C$ and $L_{C'}$ are Hamiltonian-isotopic.  
\end{corollary}
\begin{proof}
The proof is similar to that of \cite{LekiliMaydanskiy}*{Corollary 2.5}. 
\end{proof}

\begin{lemma} \label{Maslov L_C}
The Maslov index of an oriented disk in $X_n$ with boundary in $L_C$, whose boundary projects to a degree 1 cover of $C$, is $2(n-1)$. The Lagrangians $L_C$ are monotone with monotonicity constant $\tau_C = \frac{\int_{\Omega_C}\sigma}{2(n-1)}$, where $\Omega_C\subset \C$ is the region bounded by $C$ in the plane. 
\end{lemma}
\begin{proof}
We begin by considering the Lagrangian lift $L_0$ of the unit circle in the model Lefschetz fibration $\pi : \C^n \to \C$, where $\pi(z_1,\ldots,z_n) = z_1^2 + \ldots + z_n^2$. The vanishing cycle over $p\in \C\setminus \{0\}$ of a vanishing path through $p$ is 
$$
V'_{p}:= \{\sqrt{p} (x_1, \ldots, x_{n}) \, | \, (x_1, \ldots, x_{n}) \in S^{n-1}\},
$$
see \cite{SeidelBook}*{Example 16.5}. We can use the holomorphic volume form 
\begin{equation} \label{Omega}
\Omega = d z_1 \wedge \ldots \wedge d z_{n}
\end{equation}
on $\C^n$ to compute the Maslov index of a disk with boundary in $L_0$. Let $u$ be such a disk, of positive symplectic area and with boundary projecting to a simple cover of the unit circle. Let $\gamma: S^1\to L$ be a parametrization of this boundary loop such that $\pi(\gamma(t)) = e^{it}$. The imaginary part of $\left(e^{-i(nt+\pi)/2}\, \Omega\right)|_{L_0}$ vanishes, hence the Maslov index of $u$ is $n$ (see \cite{SeidelThomas} for similar computations).

To compute the Maslov class of $L_C$ in the statement of the lemma, we observe that $C$ is Lagrangian-isotopic to a connected sum $C_{-1} \# C_{1}$, where $C_{\pm 1}$ is a small simple loop around $\pm 1$ (this argument is inspired by \cite{SeidelLES}). By picking a local trivialization of the Lefschetz fibration $\pi_n$ near $\pm1$, we see that the Maslov class of $L_{C_{\pm 1}}$ can be identified with that of $L_0$ above. This implies that one can think of a disk in $X_n$ with positive symplectic area, and with boundary in $L_C$ projecting to a simple cover of $C$, as a connected sum of two disks as in the previous paragraph. Hence, the Maslov index of the disk with boundary in $L_C$ is $2n-2$, as wanted.  

The monotonicity of $L_C$ and the value of $\tau_C$ now follow from Lemma \ref{L:sigma area}. 
\end{proof}

Recall that, given a monotone Lagrangian $L$ in a symplectic manifold $(M,\omega)$ and a choice of basis $h_1,\ldots,h_m$ for the free part of $H_1(L;\mathbb \Z)$, we can define the {\em disk potential} $W_{L} : (\C^*)^m \to \C$ as
\begin{equation} \label{disk potential}
W_L(x_1,\ldots,x_m) = \sum_{u\in \M} \pm x^{\partial u},
\end{equation}
where $\M$ is the moduli space of $J$-holomorphic maps $u:(D^2,\partial D^2)\to (M,L)$ of Maslov index 2, such that $u(1)=p$, for a generic choice of point $p\in L$ and compatible almost complex structure $J$ on $(M,\omega)$. The sign associated to $u$ depends on the spin structure of $L$. If we write $\langle \partial u,h_i\rangle$ for the $h_i$-component of $[\partial u]$ in the free part of $H_1(L;\mathbb \Z)$, then $x^{\partial u}$ stands for the product $x_1^{\langle \partial u,h_1\rangle} \ldots x_m^{\langle \partial u,h_m\rangle}$. The disk potential does not depend on the choices of generic $p$ and $J$. 

\begin{lemma} \label{L:disk potential L_C}
For $n=2$, the disk potential of $L_C$ is $W_{L_C} = x_1(1+x_2)^2$, in a basis $h_1,h_2 \in H_1(L_C;\Z)\cong \Z^2$ where $h_1$ is a loop projecting to $C$ in degree 1 and $h_2$ is a fiber of the projection $\pi_2|_{L_C}$. 
The disk potential is zero if $n>2$. 
\end{lemma}
\begin{proof}
For $n=2$, the disk potential is computed in \cite{LekiliMaydanskiy}*{Lemma 2.19}, using a degeneration argument from \cite{SeidelLES}. In the proof of \cite{AurouxAnticanonical}*{Corollary 5.13}, the relevant Maslov index 2 disks are also computed explicitly, using the integrable complex structure in the target. 
The case $n>2$ follows from Lemma \ref{Maslov L_C}.
\end{proof}

Fix $\tau>0$ and a smooth embedded loop $C_\tau\subset \C\setminus \{-1,1\}$ that winds once around $-1$ and $1$ and bounds $\sigma$-area $2(n-1)\tau$. Denote by $L_\tau$, or $(S^1\times S^{n-1})_\tau$, the corresponding Lagrangian $L_{C_\tau}$. By Lemma \ref{Maslov L_C}, $L_\tau$ is monotone with monotonicity constant $\tau$. Observe that we can exhaust $\C\setminus [-1,1]$ by a collection of disjoint simple curves $C$, such that the corresponding monotonicity constants $\tau_{C}$ cover $\R_{>0}$ without repetitions. The matching sphere over the interval $[-1,1]\subset \C$ is the zero section $S^n\subset T^*S^n$.  
Assume that $C_\tau$ is the curve $C$ in Figure \ref{F_w_fig}, and denote by $F_i$ the lifts of the paths $\eta_i$ in the same figure. Similarly, denote by $F'$ the lift of the path $\eta'$.

Recall that two Lagrangian submanifolds $L,L' \subset (X,\omega)$ {\em intersect cleanly} if $K:=L\cap L'$ is a manifold and for every $x\in K$ we have $T_x K = T_xL \, \cap \, T_x L' \subset T_x X$.

\begin{lemma} \label{clean inters1}
For every $i\geq 0$, $F_i$ and $L_\tau$ intersect cleanly. For every $i,j\geq 0$, $F_i$ and $F_{j}$ intersect cleanly. Also, all these Lagrangians intersect $F'$ cleanly. 
\end{lemma}
\begin{proof}
This follows from the fact that the Lagrangians project under the map $\pi_n:X_n\to \C$ to curves that intersect transversely.  
\end{proof}

\subsection{More Lagrangians in $T^*S^3$} \label{SS:Lagrs in T*S3}

It will be useful to also consider an alternative description of the complex affine quadric 3-fold, which is symplectomorphic to $T^*S^3$. We borrow some notation from \cite{ChanPomerleanoUeda1}. Write 
$$
X = \{(z,u_1,v_1,u_2,v_2) \in \C^5 \,|\, u_1 v_1  = z +1, u_2 v_2 = z - 1 \}.
$$
Consider the Lefschetz fibrations 
\begin{align*}
\pi^i: \C^2 &\to \C \\
(u_i,v_i) & \mapsto u_i v_i + (-1)^i, 
\end{align*}
where $i\in \{1,2\}$. The map $\pi^i$ has a unique critical value at $(-1)^i$ and, given $p\in \C\setminus \{(-1)^i\}$, the vanishing circle in $(\pi^i)^{-1}(p)$ of a vanishing path through $p$ is 
$$
V_{i,p}:= \{(u_i,v_i)\in \C^2 \, | \, \pi^i(u_i, v_i) = p, |u_i|=|v_i| \}.
$$
Write also $V_{i,(-1)^i} = \{(0,0)\}$. For more details, see \cite{ChanPomerleanoUeda1} and \cite{SeidelBook}*{Example 16.5}. The affine quadric $X$ is the fiber product of these two fibrations:
\begin{displaymath}
  \xymatrix{
    & X  \ar[ld]_{f_1} \ar[rd]^{f_2} \ar@{-->}[dd]^z  \\
    \C^2 \ar[rd]_{\pi^1} &  & \C^2 \ar[ld]^{\pi^2} \\
    & \C 
  }
\end{displaymath}

The map $z:X\to \C$ is not a Lefschetz fibration, but it can be thought of as a Morse--Bott analogue, with critical values $\pm 1$ and such that the critical locus over $\pm 1$ is a copy of $\C^*$. We will consider the following analogues of the Lagrangians $L_C$ and $F_\eta$ from the previous section. It will again be useful to have Figure \ref{F_w_fig} in mind.

\begin{definition}\label{D:Lagr N}
 Given a curve $C\subset \C\setminus \{\pm 1\}$ that is the image of an embedding of $S^1$, let 
 $$
 T_C := \bigcup_{z\in C} V_{1,z}\times V_{2,z}.
 $$
Given an embedding $\eta: [0,\infty) \to \C$ such that 
 \begin{itemize}
  \item $\eta(0) = 1$,
  \item $\eta\big((0,\infty)\big) \subset \C\setminus \{\pm 1\}$ and
  \item $\eta(t) = at+b$ for some $a\in \C^*$, $b\in \C$ and $t$ large enough,
 \end{itemize}
let
\begin{align*}
N_\eta &:= \bigcup_{t\geq 0} V_{1,\eta(t)}\times V_{2,\eta(t)}.
\end{align*}
\end{definition}

Several arguments in the previous section can be adapted to this setting. This time, if $C$ encloses $\{-1,1\}$, then $T_C$ is diffeomorphic to a torus $T^3$ and we have 
$$
H_2(T^*S^3,T_C;\Z) \cong H_1(T_C;\Z) \cong \Z^3.
$$
We can pick a basis $\alpha_1$, $\alpha_2$, $\beta$ for this relative homology group, such that $\alpha_1$ is a fiber product of a Lefschetz thimble for $\pi^1$ by a point, and $\alpha_2$ is a fiber product of a point by a Lefschetz thimble for $\pi^2$.  We choose $\beta$ so that its boundary projects to a degree 1 cover of $C$. The fact that the $\alpha_i$ come from Lefschetz thimbles for the $\pi^i$ implies that they have vanishing area and Maslov index. We are left with determining the area and index of $\beta$.  

As before, there is a positive measure $\sigma'$ on $\C$, absolutely continuous with respect to the Lebesgue measure and smooth in $\C\setminus \{\pm 1\}$, such that the following result holds. 
\begin{lemma}
 $T_C$ and $N_\eta$ are Lagrangian submanifolds of $X$. The Lagrangian $T_C$ is diffeomorphic to $T^3$. Given $\beta$ as above, its $\omega$-area is $\int_{\Omega_C} \sigma'$, where $\Omega_C \subset \C$ is the region bounded by $C$, and its Maslov index is 2. Therefore, $T_C$ is monotone with monotonicity constant $\tau_C = \int_{\Omega_C} \sigma' /2$. The $N_\eta$ are Hamiltonian-isotopic to the conormal Lagrangian of the unknot in $S^3$. In particular, they are diffeomorphic to $S^1\times \R^2$ and are exact. 
\end{lemma}
\begin{proof}
The proof uses arguments similar to the ones in the previous section, so we omit them. See \cite{ChanPomerleanoUeda1} for the proofs of some of these statements. 
\end{proof}

We can also write the disk potential of $T_C$. 

\begin{lemma} \label{L:disk potential T_C}
The disk potential of $T_C$ is 
$$
W_{T_C} = x_1 (1 + x_2)(1 + x_3),
$$ 
in a basis $h_1,h_2,h_3\in H_1(T_C;\Z)\cong \Z^3$ such that $h_1$ is a loop projecting to $C$ in degree 1, $h_2 = V_{1,z}\times \{p_2\}$ for some $z\in C$ and $p_2\in V_{2,z}$, and $h_3 = \{p_1\}\times V_{2,z}$ for some $z\in C$ and $p_1\in V_{1,z}$. 
\end{lemma}
\begin{proof}
This is computed in \cite{ChanPomerleanoUeda1}.
\end{proof}

We can again exhaust $\C\setminus [-1,1]$ by disjoint simple closed curves $C$, such that the collection of monotonicity constants $\tau_C$ of the $T_C$ covers $\R_{>0}$ injectively. Fix $\tau>0$ and denote by $T^3_\tau$ the Lagrangian torus with monotonicity $\tau$ in this family. Assume that $T^3_\tau$ is the lift of the curve $C$ in Figure \ref{F_w_fig}. Denote also by $N_i$, resp.~$N'$ the lifts of the paths $\eta_i$, resp.~$\eta'$, in Figure \ref{F_w_fig}.

\begin{lemma}
For every $i\geq 0$, $N_i$ and $T^3_\tau$ intersect cleanly. For every $i,j\geq 0$, $N_i$ and $N_{j}$ intersect cleanly. All these Lagrangians intersect $N'$ cleanly.  
\end{lemma}
\begin{proof}
As in Lemma \ref{clean inters1}, this result follows from the fact that the Lagrangians project to curves in the plane that intersect transversely. 
\end{proof}

\section{Wrapped Fukaya categories}

\label{S:Fukaya categories}

The wrapped Fukaya category of a Liouville domain $M$ was introduced in \cite{AbouzaidSeidelViterbo}. In the original definition, its objects are exact Lagrangians in the completed Liouville manifold $\widehat M$. 
The Lagrangians are either compact or agree outside of a compact set with the product of $\R$ with a Legendrian submanifold of the contact manifold $\partial M$.  We will consider various versions of the wrapped Fukaya category, possibly allowing for closed monotone Lagrangians, as in \cite{RitterSmith}. 
For Lagrangians intersecting cleanly, we will use a Morse--Bott formalism similar to that of \cite{SeidelAbstractFlux} to compute the associated $A_\infty$-maps $\mu^k$. 

\begin{remark}
The monotone wrapped Fukaya category is defined in \cite{RitterSmith}, for a non-compact monotone symplectic manifold $(E,\omega)$ that is convex at infinity. In that reference, the monotonicity constant of $E$ is assumed to be strictly positive. Since the first Chern class of $T^*S^n$ vanishes, we are interested in the case of vanishing monotonicity constant. The main results in \cite{RitterSmith} still hold in this case, in particular Theorem 1.1 and equation (9.1), which we will make use of in this article.   
\end{remark}

\subsection{Coefficients}

\label{SS:Coefficients}

Some of the Floer cohomology groups we will study are defined over $\Z$, and some over a {\em Novikov field}. Given a commutative ring $R$, which for us will always be either $\Z$ or $\C$, write
$$
\K_R := \left\{ \sum_{i=0}^\infty a_i T^{\lambda_i} \,| \, a_i \in R , \lambda_i \in \R, \lambda_i< \lambda_{i+1}, \lim_{i\to \infty}\lambda_i = \infty \right\}.
$$
We will 
be mostly interested in $\K_\C$, which will be denoted simply by $\K$. We could replace $\C$ with any algebraically closed field of characteristic zero, so that the Novikov field is algebraically closed. See Section \ref{S:Generation modules} for more on this point. 

There is a {\em valuation} map
\begin{align*}
\val \colon \K_R &\to (-\infty,\infty] \\
\sum_{i=0}^\infty a_i T^{\lambda_i} &\mapsto \min \{\lambda_i \,|\, a_i \neq 0\}
\end{align*}
where $\val(0) = \infty$. Say that $\alpha \in \K_R$ is {\em unitary} if $\val(\alpha) = 0$. Denote by 
$U_{\K_R}^* := \left\{ \alpha \in \K_R \,|\, \val(\alpha) = 0 \text{ and $\alpha$ is invertible}\right\}$ 
the group of invertible unitary elements in $\K_R$ (if $R$ is a field, $\val(\alpha)=0$ implies that $\alpha$ is invertible), by 
$\K_{R,0} := \left\{ \alpha \in \K_R \,|\, \val(\alpha) \geq 0\right\}$ 
the {\em Novikov ring} and by 
$\K_{R,+} := \left\{ \alpha \in \K_R \,|\, \val(\alpha) > 0\right\}$ 
the maximal ideal in $\K_{R,0}$.

\subsection{Morse--Bott Floer cohomology for clean intersections}

We will use a Morse--Bott version of the Fukaya category, where Lagrangians are allowed to intersect cleanly, as in \cite{SeidelAbstractFlux}*{Section 3.2}. For more details on a construction of the Fukaya category with a Morse--Bott definition of the $A_\infty$-algebra of endomorphisms of a Lagrangian submanifold, see \cite{SheridanPOP}*{Section 4}. These references assume that the Lagrangians are exact, which precludes disk bubbling. Lemma \ref{L:vanishing potential} below guarantees that if $(L,\xi)$ is a Lagrangian with a rank 1 unitary local system giving a non-trivial object in the Fukaya category, then $\xi$ corresponds to a zero of the disk potential of $L$. This is useful when considering 2-parameter families of pearly trees of holomorphic disks (to prove the $A_\infty$-relations, for instance), since the vanishing of the disk potentials implies the cancellation of configurations with disk bubbles. Therefore, we can assume for many purposes that the relevant Lagrangians bound no holomorphic disks. For more on the Floer cohomology of cleanly intersecting Lagrangians, see for example \cite{PozniakThesis}, \cite{FrauenfelderThesis}, \cite{FOOO}, \cite{BiranCorneaUniruling}, and \cite{SchmaeschkeClean}.

Let us briefly define the relevant Floer complexes. Let $L_0,L_1$ be two Lagrangians such that each $L_i$ is equipped with:
\begin{itemize}
 \item an orientation and a spin structure;  
 \item a local system $\xi$ 
on a trivial $\K_R$-bundle $E_i = \oplus_k E_{i,k}$, where the direct sum is finite and the summand $E_{i,k}$ of grading $k$ is a finite rank trivial vector bundle over $L_i$, in degree $k$. If the rank of $E_i$ is 1, then the local system can be unitary (with holonomy taking values in $U_{\K_R}^*$). If the rank of $E_i$ is bigger than 1, then the local system has trivial holonomy. 
 \end{itemize}

\begin{remark}
In this article, the zero section in $T^*S^n$, with $n>1$ even, is the only class of Lagrangians that we will equip with local systems of rank greater than 1. 
\end{remark}

\begin{remark}
The choices of spin structures on the Lagrangians are necessary to orient moduli spaces of holomorphic curves. Nevertheless, in our computations we will not be very careful in specifying spin structures on the Lagrangians. This is because the effect of changing the spin structure on a Lagrangian is a change in signs associated to holomorphic curves, and this change can be compensated by the choice of a different local system $\xi$ on the Lagrangian. 
\end{remark}

The categories of exact Lagrangians we will consider are $\Z$-graded, so we will need additional choices of gradings for their objects (as in \cite{SeidelBook}*{Section 12a}). 
The categories of monotone Lagrangians will only be $\Z/2\Z$-graded, so we will not need to choose gradings in that case.

Let $L_0,L_1$ be Lagrangians as above, with local systems $\xi_i$ on $\K_R$-bundles $E_i$ over $L_i$. Write $\mathcal L_i\coloneqq (L_i,\xi_i)$. 
Assume that the Lagrangians intersect cleanly and let $f: L_0 \cap L_1 \to \R$ be a Morse function. In accordance with \cite{SeidelAbstractFlux}*{Equation (3.9)}, we define the cochain complex
$$
CF^*(\mathcal L_0,\mathcal L_1) := \bigoplus_{C\subset L_0\cap L_1} \qquad \bigoplus_{\mathclap{\substack{p\in \crit(f|_C)}}} \, \big(\Hom_{\K_R}\left((E_0)_p,(E_1)_p\right) \otimes_{\Z} \mathfrak o_C \big)[-\deg(p)]
$$ 
where the $C\subset L_0\cap L_1$ are connected components of the intersection, $\mathfrak o_C$ is the orientation line of $C$ (a rank 1 local system over $\Z$ depending on the spin structures of the $L_i$), and 
$\Hom_{\K_R}$ denotes $\K_R$-linear maps. 
The Floer degree associated to the critical point $p$ is $\deg(p) = \dim(C)- \ind(p) + \deg(C)$, where $\ind(p)$ is the Morse index of $p$ as a critical point of $f|_C$ and $\deg(C)$ is an absolute Maslov index, which depends on the gradings of the $L_i$.%
\footnote{With this choice of grading, the graded vector space $CF^*(\mathcal L_0,\mathcal L_1)$ can be identified with $CF^*(\varphi(\mathcal L_0),\mathcal L_1)$, for a Hamiltonian isotopy $\varphi$ supported near $L_0\cap L_1$ and constructed from $f$.}
Recall that according to the cohomological convention that we use, $[-k]$ adds $k$ to the degree.

The operations $\mu^k$ are defined on tensor products of these chain complexes, via counts of {\em pearly trees}. We give a very brief description of these, referring the reader to \cite{SeidelAbstractFlux}*{Section 3.2} for more details. Given a collection $\mathcal L_0, \ldots, \mathcal L_k$ of Lagrangians with local systems, a pearly tree contributing to 
$$
\mu^k: CF^*(\mathcal L_{k-1},\mathcal L_k) \otimes \ldots \otimes CF^*(\mathcal L_0,\mathcal L_1) \to CF^*(\mathcal L_0,\mathcal L_k)
$$
is a collection of perturbed pseudoholomorphic disks (with respect to auxiliary almost complex structures and perturbing 1-forms) with boundary punctures and Lagrangian boundary conditions, connected by gradient flow lines of auxiliary Morse functions and metrics on the clean intersections of the $L_i$. This collection of disks and flow lines can be concatenated into a continuous map from a disk with $n+1$ boundary punctures to the symplectic manifold, with boundary components of the disk mapping to the Lagrangians $L_0, \ldots, L_k$, see Figure \ref{pearly_tree_fig}. The contribution of a rigid configuration of disks and flow lines to $\mu^k$ is determined by the areas of the pseudoholomorphic disks (which are encoded in the exponents of the variable $T$ in the Novikov field), by signs specified by the spin structures on the $L_i$, and by parallel transport with respect to the local systems $\xi_i$ on the $E_i$ along the boundary components of the concatenated disk (with the input elements of $\Hom_{\K_R}(E_i,E_{i+1})$ applied at the boundary punctures).
The $\mu^k$ satisfy the $A_\infty$-relations, which can be written in abbreviated form as $\mu\circ\mu=0$.

\begin{figure}
  \begin{center}
    \def\svgwidth{0.4\textwidth}
    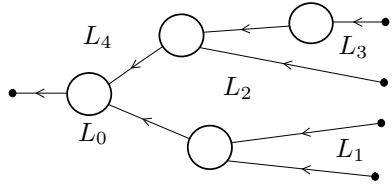
  \end{center}
  \caption{A pearly tree contributing to $\mu^4$}
  \label{pearly_tree_fig}
\end{figure}

We will also consider Fukaya categories containing additional objects. A {\em bounding cochain} on an object $\mathcal L$ in a $\Z/2\Z$-graded Fukaya category is $b\in CF^{\odd}(\mathcal L,\mathcal L)$ satisfying the {\em Maurer--Cartan equation} 
\begin{equation} \label{MC}
\sum_{k=1}^\infty \mu^k(b,\ldots,b) = 0, 
\end{equation}
see \cite{FOOO} (to ensure that all the $A_\infty$-operations converge, we assume that $b$ has strictly positive valuation). We can enlarge our category by allowing objects of the form $(\mathcal L,b)$, for such $b$. The object $\mathcal L$ can be identified with $(\mathcal L,0)$. Given objects $\hat{\mathcal L}_0 = (\mathcal L_0,b_0),\ldots,\hat{\mathcal L}_k = (\mathcal L_k,b_k)$ in the enlarged category, the $A_\infty$-maps 
$$
\hat\mu^k : CF^*(\hat{\mathcal L}_{k-1},\hat{\mathcal L}_{k})\otimes \ldots \otimes CF^*(\hat{\mathcal L}_0,\hat{\mathcal L}_{1}) \to CF^*(\hat{\mathcal L}_0,\hat{\mathcal L}_{k})
$$
are given by 
$$
\hat\mu^k(x_k,\ldots,x_1) := \sum_{l_0,\ldots,l_k \ge 0} \mu^{(k+\sum_i l_i)}(\underbrace{b_k,\ldots,b_k}_{l_k},x_k,b_{k-1},\ldots,b_1,x_1,\underbrace{b_0,\ldots,b_0}_{l_1}).
$$
The fact that the $b_i$ satisfy the Maurer--Cartan equation \eqref{MC} implies the $A_\infty$-equations $\hat\mu \circ \hat \mu = 0$. Since $\hat \mu^k$ agrees with $\mu^k$ when all $b_i=0$, we will continue to write $\mu^k$ instead of $\hat\mu^k$.

\subsection{Wrapped Floer cohomology}

We will use a model for wrapped Floer cohomology from \cite{AbouzaidSeidelFuture}, presented in \cite{GPS1} and \cite{SylvanFunctors}. Let $L_0$ be a non-compact Lagrangian, which in this paper will be either a cotangent fiber $F$ or a conormal Lagrangian of the unknot $N\subset T^*S^3$. We pick a family $L_i$ of Lagrangians that are lifts of paths $\eta_i$ in the base of the Lefschetz fibration $\pi_n$ from Section \ref{SS:Lagrs in T*Sn} (in the case of $F$), or in the base of the fiber product of Lefschetz fibrations $\pi^i$ from Section \ref{SS:Lagrs in T*S3} (in the case of $N$), where the path $\eta_i$ wraps $i$ times around the two critical values, see Figure \ref{F_w_fig}. Then, given another Lagrangian $L'$, we have 
$$
HW^*(L_0,L') := \lim_{i\to \infty} HF^*(L_i,L'),
$$
with the limit taken with respect to the continuation maps relating $L_i$ and $L_{i+1}$. For the equivalence of this model with the usual definitions involving fast growing Hamiltonians, see \cite{GPS1}*{Lemma 3.37} and \cite{SylvanFunctors}*{Proposition 2.6}. In these references, the wrapped Fukaya category is defined by localizing the Fukaya category on the continuation maps that were just mentioned. We will combine this approach to wrapped Floer cohomology with the definition of Morse--Bott Floer cohomology above, where the Lagrangians intersect cleanly and are possibly equipped with local systems and bounding cochains.

\subsection{Wrapped Fukaya categories} \label{S:W}

We will consider several versions of the Fukaya $A_\infty$-category of $T^*S^n$. Recall that $R$ is either $\Z$ or $\C$.

\begin{itemize}
\item $\W^\Z(T^*S^n;\Z)$ is a category whose objects are either the $F_\eta$ from Definition \ref{D:Lagr F} or compact oriented exact Lagrangians. When $n=3$, we also include the objects $N_\eta$ from Definition \ref{D:Lagr N}. Objects are equipped with $\Z$-gradings and spin structures. Morphism spaces are wrapped Floer cochain complexes with coefficients in $\Z$. The differential and higher $A_\infty$-operations count rigid pearly trees, without keeping track of areas (which can be thought of as setting $T=1$ in the Novikov field $\K_\Z$). 
 
 \item $\W^\Z(T^*S^n;\K_R)$ has the same objects as $\W^\Z(T^*S^n;\Z)$. The difference is that the morphism spaces are now wrapped Floer cochain complexes {\em with coefficients in $\K_R$}, to keep track of the symplectic areas of the disks in the pearly trees that contribute to the $A_\infty$ operations. 

\item $\W^{\Z/2\Z}(T^*S^n;\K_R)$ is obtained from $\W^\Z(T^*S^n;\K_R)$ by collapsing the $\Z$-gradings to $\Z/2\Z$-gradings. If $n$ is odd, allow also objects of the form $(S^n,b_\alpha)$, where $S^n$ is the zero section and $b_\alpha=\alpha [pt]$ is a bounding cochain with $\alpha \in \K_{R,0}$ and $[pt] \in H^n(S^n;\K_R)$. See Remark \ref{R:alpha in K_0} below for why we impose $\alpha\in \K_{R,0}$. We have implicitly chosen a perfect Morse function on $S^n$, and $[pt]$ is given by the minimum of that function (the maximum yields the unit in the $A_\infty$-algebra of $S^n$). Since $S^n$ bounds no disks, it is clear that all the summands in \eqref{MC} vanish for $b=b_\alpha$. 
If $n$ is even, we want to allow instead objects corresponding to bounding cochains in $S^n\oplus S^n[1]$ (summing and shifting objects are allowed in the additive enlargement of the Fukaya category). We implement this by equipping $S^n$ with the trivial graded $\K_R$-bundle $E:=\K_R\oplus \K_R[1]$, and bounding cochains $b_{\alpha,\beta}\in H^{odd}(S^n;\End(E))$ of the form 
$$
b_{\alpha,\beta} \coloneqq \begin{pmatrix}
             0 & \beta \\
             \alpha & 0
            \end{pmatrix}_{[pt]},
$$
where $\alpha,\beta\in \K_{R,0}$ and the matrix represents an endomorphism of the fiber of $E$ at the minimum of the auxiliary perfect Morse function on $S^n$. 
 
 \item $\W^{\Z/2\Z}_{\mon}(T^*S^n;\K_R)$ is an extension of $\W^{\Z/2\Z}(T^*S^n;\K_R)$, 
 including closed monotone Lagrangians. The objects are equipped with orientations 
 and spin structures, and are $\Z/2\Z$-graded. We also equip monotone Lagrangians with unitary rank 1 
 local systems over $\K_R$. 
The construction of the monotone wrapped Fukaya category is given in \cite{RitterSmith}. 
See also \cite{SheridanFano} for a definition of the monotone Fukaya category in a closed setting.

 \item $\F^{\Z/2\Z}_{\mon}(T^*S^n;\K_R)$ is the full subcategory of $\W^{\Z/2\Z}_{\mon}(T^*S^n;\K_R)$ containing only those objects whose underlying Lagrangians are closed. 
\end{itemize}

It is an important fact that in all these categories the isomorphism class of objects is preserved by Hamiltonian isotopies; in the presence of bounding cochains, this means that if $b$ is a bounding cochain on $L$, and $L'$ is Hamiltonian-isotopic to $L$, then there is a bounding cochain $b'$ on $L'$ so that the two corresponding objects of the Fukaya category are isomorphic.

\begin{remark} \label{R:alpha in K_0}
  If we equip Lagrangians with bounding cochains valued in the maximal ideal $\K_{R,+}$ of $\K_{R,0}$, then we are guaranteed convergence of all the $A_\infty$-operations deformed by such bounding cochains. In our case, since 
the degree of $[pt]\in H^*(S^n;\Z)$ is $n>1$, we could in fact allow bounding cochains $\alpha [pt]$ for arbitrary $\alpha\in \K_R$ in the category of exact Lagrangians $\W^{\Z/2\Z}(T^*S^n;\K_R)$. 
  Nevertheless, for bounding cochains $\alpha [pt]$ with $\val(\alpha) < 0$, we could run into convergence issues when taking morphisms with monotone Lagrangians in $\W^{\Z/2\Z}_{\mon}(T^*S^n;\K_R)$, which is why we restrict to bounding cochains with coefficients in $\K_{R,0}$.  
With minor modifications to our arguments, we could also have equipped all objects in $\W^{\Z/2\Z}_{\mon}(T^*S^n;\K_R)$ 
with finite rank unitary local systems and suitable bounding cochains.
\end{remark}

Observe that we can define several functors between these categories:
\begin{itemize}
 \item $\mathcal G_1 \colon \W^\Z(T^*S^n;\Z) \to \W^\Z(T^*S^n;\K_R)$ is the identity on objects. Fix a primitive $f_L$ for every exact Lagrangian $L$. Given exact Lagrangians $L_0,L_1$ in $\W^\Z(T^*S^n;\Z)$, map $x\in CW^*(L_0,L_1;\Z)$ to $T^{f_1(x)-f_0(x)}x \in CW^*(L_0,L_1;\K_R)$, where $f_i :=f_{L_i}$.  
 If $u$ is a pearly tree contributing to $\mu^k$ in $\W^\Z(T^*S^n;\Z)$, then 
 the contribution of $u$ to $\mu^k$ in $\W^\Z(T^*S^n;\K_R)$ is weighted by the factor $T^{\int_{D^2} u^*\omega}$, where the integral is over all the holomorphic disks in the pearly tree. The functor $\mathcal G_1$ depends on the choices of primitives $f_L$, but different choices yield isomorphic functors (we could eliminate this choice by incorporating the primitives in the definition of objects of the exact Fukaya category).

 \item $\mathcal G_2 \colon \W^\Z(T^*S^n;\K_R) \to \W^{\Z/2\Z}(T^*S^n;\K_R)$ is given by collapsing the $\Z$-grading to a $\Z/2\Z$-grading, followed by inclusion of objects.

 \item $\mathcal G_3 \colon \W^{\Z/2\Z}(T^*S^n;\K_R) \to \W^{\Z/2\Z}_{\mon}(T^*S^n;\K_R)$ is given by inclusion of objects, as are $\mathcal G_4 \colon \W^{\Z/2\Z}_{\mon}(T^*S^n;\K_\Z) \to \W^{\Z/2\Z}_{\mon}(T^*S^n;\K)$ (recall that $\K=\K_\C$) and $\mathcal G_5 \colon \F^{\Z/2\Z}_{\mon}(T^*S^n;\K_\Z) \to \W^{\Z/2\Z}_{\mon}(T^*S^n;\K)$.
\end{itemize}

\begin{remark}
Let $L$ be a monotone Lagrangian. A unitary rank 1 local system $\xi$ on the trivial $\K_R$-bundle over $L$ can be specified by a homomorphism (the holonomy)
$$
\hol_\xi : H_1(L;\Z) \to U_{\K_R}^*.
$$
If, in the definition \eqref{disk potential} of the disk potential $W_L$, we replace $x^{\partial u}$ with $\hol_\xi(\partial u)$, then we get an element of $\K_R$ that we denote by $W_L(\xi)$. 

When defining the monotone category $\W^{\Z/2\Z}_{\mon}(T^*S^n;\K_R)$, one can only take morphisms between objects $(L_1,\xi_1)$ and $(L_2,\xi_2)$ if $W_{L_1}(\xi_1)=W_{L_2}(\xi_2)$, see \cite{OhMonotoneI}. This does not impose an additional constraint in our case, due to the following result. It can be interpreted as saying that the monotone Fukaya category of $T^*S^n$ is unobstructed. 
\end{remark}

\begin{lemma} \label{L:vanishing potential}
Let $L\subset T^*S^n$ be a compact monotone Lagrangian with a unitary local system $\xi$ on a trivial line bundle. 
Write $\mathcal L = (L,\xi)$. If $HF^*(\mathcal L,\mathcal L;\K_R) \neq 0$, then $W_L(\xi)=0$. 
\end{lemma}
\begin{proof}
This follows from \cite{RitterSmith}*{equation (9.1)} (see \cite{RitterSmith}*{Section 3.17} for a discussion of Lagrangians with local systems). According to that equation, if $HF^*(\mathcal L,\mathcal L;\K_R) \neq 0$ then $m_0(\mathcal L) = W_L(\xi)$ can be computed by applying the open-closed map to the first Chern class of the total space of $T^*S^n$. The lemma follows from the vanishing of this class.  
\end{proof}

\begin{remark} \label{R:crit pts of potentials}
Let $L^n$ be a monotone Lagrangian torus with disk potential $W_L$. The critical points of $W_L$ in $(U_{\K_R}^*)^n$ correspond to the rank 1 unitary local systems $\xi$ on the trivial $\K_R$-line bundle over $L$ for which $HF^*(\mathcal L,\mathcal L;\K_R)\neq 0$, where $\mathcal L = (L,\xi)$, see \cite{SheridanFano}*{Proposition 4.2}. Recall that $HF^*(\mathcal L,\mathcal L;\K_R)\neq 0$ is equivalent to the non-triviality of $\mathcal L$ in $\mathcal F^{\Z/2\Z}_{\mon}(T^*S^n;\K_R)$. By Lemma \ref{L:disk potential L_C}, the Lagrangian tori $L_C\subset T^*S^2$ have disk potential $W_1=x_1(1+x_2)^2$. The critical locus of this potential is given by the condition $x_2=-1$.
Recall also that Lemma \ref{L:disk potential T_C} says that the disk potential of a Lagrangian torus $T_C\subset T^*S^3$ is $W_2=x_1(1+x_2)(1+x_3)$, whose critical locus is given by $x_2=x_3=-1$. 
Observe that, for both tori $L_C\subset T^*S^2$ and $T_C\subset T^*S^3$, the disk potentials vanish on their critical points, which is compatible with Lemma \ref{L:vanishing potential}.
\end{remark}

\subsection{Yoneda functors} \label{SS:Yoneda}

In this section we will be working over the field $\K_\C = \K$, since we will use some formality results from Section \ref{S:Formality algebra}. 
Let $\A_{\K} := CW^*(F,F;\K)$ be the $A_\infty$-algebra of a cotangent fiber in $T^*S^n$, with $n\geq 2$, and let $\mmod^{A_\infty}(\A_{\K})$ be the differential $\Z/2\Z$-graded category of right $A_\infty$-modules over $\A_{\K}$. Given two objects $\M$ and $\M'$ in $\mmod^{A_\infty}(\A_{\K})$, the morphism space $\hom_{\mmod^{A_\infty}(\A_{\K})}(\M,\M')$ is a chain complex computing $\Ext_{\A_{\K}}^*(\M,\M')$, see \cite{SeidelBook}*{Remark 2.15}.

There is a Yoneda functor 
\begin{align*}
\Y : \W^{\Z/2\Z}_{\mon}(T^*S^n;\K) &\to \mmod^{A_\infty}(\A_{\K}) \\
\mathcal L & \mapsto CW^*(F,\mathcal L)
\end{align*}
which restricts to a functor
$$
\Y_c : \F^{\Z/2\Z}_{\mon}(T^*S^n;\K) \to \mmod^{A_\infty}_{pr}(\A_{\K}),
$$
where $\mmod^{A_\infty}_{pr}(\A_{\K}) \subset \mmod^{A_\infty}(\A_{\K})$ is the subcategory of {\em proper modules} $\M$, such that $H^*(\M)$ is finite dimensional over ${\K}$ (the subscript in $\Y_c$ stands for `compact').

Now, let  $A_{\K} := H^*(\A_{\K})$ be the cohomology algebra of $\A_{\K}$. Let $\mmod(A_{\K})$ be the $\Z/2\Z$-graded category of right $A_{\K}$-modules, such that morphism spaces are 
$\Ext_{A_{\K}}^*$ groups (respecting the $\Z/2\Z$-gradings). There is a functor
\begin{align*}
H:\mmod^{A_\infty}(\A_{\K}) &\to \mmod(A_{\K}) \\
\M & \mapsto H^*(\M)
\end{align*}
which restricts to
$$
H_c:\mmod^{A_\infty}_{pr}(\A_{\K}) \to \mmod_{pr}(A_{\K}) \\
$$
where $\mmod_{pr}(A_{\K}) \subset \mmod(A_{\K})$ is the subcategory of finite dimensional $\Z/2\Z$-graded modules over $A_{\K}$.  

Proposition \ref{P:N split-generates} below implies that the functor $\Y$ (hence also $\Y_c$) is cohomologically full and faithful. According to Corollary \ref{C:mod(A) is formal} below, $H$ (hence also $H_c$) is a quasi-equivalence of categories. We conclude the following.
\begin{proposition} \label{C:Yoneda ff}
The composition
\begin{align*}
Y:=H\circ \Y : \W^{\Z/2\Z}_{\mon}(T^*S^n;\K) &\to \mmod(A_{\K}) \\
\mathcal L & \mapsto HW^*(F,\mathcal L)
\end{align*}
and its restriction
$$
Y_c:=H_c\circ \Y_c : \F^{\Z/2\Z}_{\mon}(T^*S^n;\K) \to \mmod_{pr}(A_{\K}) 
$$
are cohomologically full and faithful embeddings.  \qed
\end{proposition}

\section{Floer cohomology computations}

\label{S:HF computations}

\subsection{The Lagrangians $F$ and $N$}

Recall from Section \ref{S:construct Lagrangians} that the Lagrangian lifts $F_\eta\subset T^*S^n$ and $N_\eta\subset T^*S^3$ are Hamiltonian-isotopic to, respectively, a cotangent fiber (which we denote by $F$) and the conormal Lagrangian of an unknot in $S^3$ (which we denote by $N$). 

\begin{proposition} \label{P:N split-generates}
The cotangent fiber $F$ generates $\W^\Z(T^*S^n;\Z)$ and $\W^\Z(T^*S^n;\K_R)$, and it split-generates $\W^{\Z/2\Z}_{\mon}(T^*S^n;\K_R)$. When $n=3$, the Lagrangian $N$ split-generates $\W^\Z(T^*S^3;\Z)$, $\W^\Z(T^*S^3;\K_R)$ and $\W^{\Z/2\Z}_{\mon}(T^*S^3;\K_R)$.
\end{proposition}
\begin{proof}
The fact that a cotangent fiber generates $\W^\Z(T^*S^n;\Z)$ is proven in \cite{AbouzaidCotangentFiber}, and the result follows for $\W^\Z(T^*S^n;\K_R)$. Let us recall the argument: it is first shown that a cotangent fiber split-generates, and this is then extended to a proof of generation. The fact that a cotangent fiber split-generates follows from combining the split-generation criterion of \cite{AbouzaidGeneration} with results about cotangent bundles and loop spaces in \cite{AbouzaidBasedLoops}. The split-generation criterion of \cite{AbouzaidGeneration} is extended to the monotone wrapped Fukaya category in \cite{RitterSmith}*{Theorem 1.1}, and can again be combined with results in \cite{AbouzaidBasedLoops} to conclude that a cotangent fiber split-generates $\W^{\Z/2\Z}_{\mon}(T^*S^n;\K_R)$. 

The previous paragraph and Lemma \ref{HW(N)} below imply the result for $N$. 
\end{proof}

\begin{remark}
It would be interesting to determine if a cotangent fiber generates $\W^{\Z/2\Z}_{\mon}(T^*S^n;\K_R)$, but this is not necessary to prove the results in this article about compact monotone Lagrangians. 
\end{remark}

Recall that under our assumption that $n\geq 2$ we have $HW^*(F,F;\Z) \cong \Z[u]$, where $\deg(u)=1-n$, as follows from \cite{AbouzaidBasedLoops}. Denote this ring by $A_\Z$. Also denote by $F_0$ a cotangent fiber corresponding to a lift of a path $\eta_0$ through the critical value $1$ of $\pi_n$, and by $F'$ one that is a lift of a path $\eta'$ through $-1$, see Figure \ref{F_w_fig}. Since $F_0$ and $F'$ are Hamiltonian-isotopic, we have  
\begin{equation} \label{HF(F,F')}
HW^*(F_0,F';\Z) \cong A_\Z.
\end{equation}
On the other hand, 
$$
HW^*(F_0,F';\Z) \cong \lim_{i\to \infty} HF^*(F_i,F';\Z),
$$
where the $F_i$ are lifts of the paths $\eta_i$ illustrated in Figure \ref{F_w_fig}. In our Morse--Bott model, the cochain complex for $CF^*(F_i,F';\Z)$ is described, as a graded abelian group, as
$$
\bigoplus_{k=0}^i H^{*+(n-1)(1+2k)}(S^{n-1};\Z).
$$ 

\begin{lemma} \label{lemma: continuation maps}
The chain level continuation maps
$$
CF^*(F_i,F';\Z) \to CF^*(F_{i+1},F';\Z) 
$$
are inclusions, and the differentials vanish on these chain complexes. 
\end{lemma}
\begin{proof}
There is a compactly supported isotopy of $\C\setminus \{\pm 1\}$ taking the path $\eta_{i+1}$ to a path $\tilde \eta_{i+1}$ such that:
\begin{itemize}
\item $\tilde \eta_{i+1}$ intersects $\eta'$ transversely at precisely $i+1$ points (just like $\eta_{i+1}$) and 
\item $\tilde \eta_{i+1}$ coincides with $\eta_i$ in a connected portion of the path starting at $1\in \C$ and containing all the intersections of $\eta_i$ with $\eta'$. 
\end{itemize}
This isotopy of the base can be lifted to a compactly-supported Hamiltonian isotopy of $T^*S^n$, taking $F_{i+1}$ to a Lagrangian lift $\tilde F_{i+1}$ of the path $\tilde \eta_{i+1}$. The continuation map 
$$
CF^*(F_i,F';\Z) \to CF^*(\tilde F_{i+1},F';\Z) 
$$
is clearly an inclusion, and so must be the map in the statement. 

As for the vanishing of the differentials, notice that for degree reasons this is only a non-trivial statement when $n=2$. The fact that 
$$
A_\Z \cong \lim_{i\to \infty} HF^*(F_i,F';\Z)
$$
and that the continuation maps are inclusions implies that there can be no non-trivial differentials in the chain complexes $CF^*(F_i,F';\Z)$. This is because the direct limit would not have the correct rank in the degrees related by a non-trivial contribution to the differential. 
\end{proof}

The previous result implies the following. 

\begin{lemma} \label{u in F}
Up to a factor $\pm 1$, the unit $e$, resp.~the generator $u$, in $A_\Z$ is represented by the minimum, resp.~maximum, of the auxiliary Morse function on $F_0\cap F' \cong S^{n-1}$, thought of as a class in $HF^{0}(F_0,F';\Z)$, resp.~$HF^{1-n}(F_0,F';\Z)$. 
\end{lemma}

We now consider the Lagrangian $N$.  

\begin{remark}
In the following, we use the cohomological degree shift notation, where $[k]$ corresponds to a shift by $-k$. 
\end{remark}

\begin{proposition}\label{HW(N)}
The Lagrangian $N \in T^*S^3$ is quasi-isomorphic to $F \oplus F[1]$ in $\W^\Z(T^*S^3;\Z)$. 
In particular, $HW^*(N,N;\Z)$ is isomorphic to the graded matrix algebra
$$
B_\Z:=\begin{pmatrix}
 A_\Z & A_\Z[1] \\
 A_\Z[-1] & A_\Z
\end{pmatrix}.
$$
Hence, $HW^*(N,N;\K_R)$ is isomorphic to $B_\Z\otimes_\Z \K_R$ for any commutative ring $R$. 
\end{proposition}

\begin{proof}
Recall the construction, in \cite{AbouzaidBasedLoops}, of a cohomologically fully faithful $A_\infty$-functor
$$
\mathcal F : \mathcal W^\Z(T^*Q;\Z) \to \Tw(\mathcal P(Q)),
$$
where the target is a category of twisted complexes on a Pontryagin category $\mathcal P(Q)$ of a closed spin manifold $Q$. Objects in $\mathcal P(Q)$ are points in $Q$, with $\hom_{\mathcal P(Q)}(q_1,q_2) = C_{-*}(\Omega_{q_1,q_2}(Q);\Z)$ (cubical chains) and composition defined via concatenation of paths.  Here, $\Omega_{q_1,q_2}(Q)$ is the space of Moore paths in $Q$ that start at $q_1$ and end at $q_1$. Write also $\Omega_{q}$ for $\Omega_{q,q}$. Given an object in $\mathcal W^\Z(T^*Q;\Z)$, which is a $\Z$-graded exact spin Lagrangian $L$ in $T^*Q$, we can assume (up to a Hamiltonian isotopy) that $L$ intersects the zero-section transversely at the points $q_1,\ldots,q_m$. The image of $L$ under $\mathcal F$ is a twisted complex supported on a direct sum of grading shifts of the $q_i$. The differential in the twisted complex is constructed from moduli spaces of Floer strips between $Q$ and $L$. 

Let us use this functor in our setting. The Lagrangian $N$ intersects $S^3$ cleanly along a copy of $S^1$. One can deform $N$ by a Hamiltonian isotopy so that it intersects $S^3$ transversely at exactly two points $q_1$ and $q_2$, with consecutive indices. Hence, $\mathcal F(N)$ is a twisted complex supported on the sum of shifts of $q_1$ and of $q_2$. The differential on this twisted complex is given by a cycle in $C_{0}(\Omega_{q_1,q_2}(S^3);\Z)$. Homologous cycles yield quasi-isomorphic twisted complexes, so the differential on $\mathcal F(N)$ is determined by an element $x\in H_{0}(\Omega_{q_1,q_2}(S^3);\Z) \cong \Z$. 

Given $q\in S^3$ and identifying $H_{-*}(\Omega_{q}(S^3);\Z)$ with $HW^*(F_q,F_q;\Z)$, we can say that $N$ is quasi-isomorphic to $\Cone(F_q\stackrel{x}{\to} F_q)$ in a category of twisted complexes over $\mathcal W^\Z(T^*Q;\Z)$, where $x$ is now thought of as a class in $HW^0(F_q,F_q;\Z) \cong \Z$. In particular, up to a degree shift, 
\begin{align*}
HF^*(N,S^3;\Z) &\cong H^*\left(\Cone(HF^*(F_q,S^3;\Z) \stackrel{x}{\to} HF^*(F_q,S^3;\Z)) \right) \cong \\
&\cong H^*\left(\Cone(\Z \stackrel{x}{\to} \Z)\right) \cong \begin{cases}
\Z[1]\oplus \Z & \text{ if } x=0 \\
(\Z/x \Z) & \text{ otherwise} 
\end{cases}. 
\end{align*}

On the other hand, one can adapt \cite{PozniakThesis}*{Proposition 3.4.6} to Floer cohomology with $\Z$-coefficients (instead of $\Z/2\Z$), and conclude that 
$$
HF^*(N,S^3;\Z) \cong H^*(S^1;\Z),
$$
up to a degree shift. Therefore, we conclude that $x=0$, the differential in the twisted complex $\mathcal F(N)$ is trivial, and $N$ is quasi-isomorphic to $F\oplus F[1]$, as wanted. 
\end{proof}

\begin{remark}
Strictly speaking, the argument above only implies that $N=F\oplus F[1]$ up to a global degree shift. However, this will be enough for our purposes, since the main application of the previous proposition will be in Lemma \ref{HF(N,T)}, which is about the $\Z/2\Z$-graded monotone Fukaya category. 
\end{remark}

The ring $B_\Z$ of endomorphisms of $F\oplus F[1]$ can be represented pictorially as follows:
\tikzset{node distance=2cm, auto}
\begin{center}
\begin{tikzpicture}
  \node (A) {$F$};
  \node (B) [right of=A] {$F[1]$};
  \draw [->,out=45,in=135,looseness=0.75] (A.north) to node[above]{$%
  \begin{pmatrix}
  0 & * \\
  0 & 0
  \end{pmatrix}%
  $}  (B.north);
  \draw [->,out=-135,in=-45,looseness=0.75] (B.south) to node[below]{$%
  \begin{pmatrix}
  0 & 0 \\
  * & 0
  \end{pmatrix}%
  $} (A.south);

  \path (A) edge [->,in=160, out = 200, loop] node[left] {$%
  \begin{pmatrix}
  * & 0 \\
  0 & 0
  \end{pmatrix}%
  $} (A);
  \path (B) edge [->,in=340, out = 15, loop] node[right] {$%
  \begin{pmatrix}
  0 & 0 \\
  0 & *
  \end{pmatrix}%
  $} (B);
  \end{tikzpicture}
\end{center}
Define 
$$
e_1 := \begin{pmatrix}
 1 & 0 \\
 0 & 0
\end{pmatrix}, \,
e_2 := \begin{pmatrix}
 0 & 0 \\
 0 & 1
\end{pmatrix}, \,
e_{21} := \begin{pmatrix}
 0 & 0 \\
 1 & 0
\end{pmatrix}, \,
e_{12} := \begin{pmatrix}
 0 & 1 \\
 0 & 0
\end{pmatrix}.
$$
Note that 
$$
|e_1| = 0 = |e_2|, \, |e_{21}| = 1, \, |e_{12}| = -1. 
$$
As a graded free abelian group, $B_\Z$ has generators in low degrees given by
$$
\begin{tabular}{c|c|c|c|c|c}
 degree & 1 & 0 & $-1$ & $-2$ & $-3$ \\
\hline
 generator & $e_{21}$ & $e_1, e_2$ & $e_{12}, u e_{21}$ & $u e_1, u e_2$ & $u e_{12}, u^2 e_{21}$ 
\end{tabular}
$$

\

In a manner similar to what we did above for $F$, let us give a more explicit description of the Morse--Bott wrapped Floer cohomology of $N$. Denote by $N'$ the lift of a path $\eta'$ through $-1$ and by the $N_i$ lift of a path $\eta_i$ through $1$ that winds $i$ times around the critical values of the Morse--Bott Lefschetz fibration, see Figures \ref{F_w_fig} and \ref{N0N1_fig}.

By Proposition \ref{HW(N)}, we know that, 
$$
B_\Z\cong HW^*(N_0,N';\Z) \cong \lim_{i\to \infty} HF^*(N_i,N';\Z).
$$
The Morse--Bott Floer cochain complex for $CF^*(N_i,N';\Z)$ with $i\geq 0$ is given, as a graded abelian group, by
$$
\bigoplus_{k=0}^i H^{*+2k+1}(T^2;\Z).
$$ 
Similarly to what is stated in Lemma \ref{lemma: continuation maps} for $F$, the continuation maps
$$
CF^*(N_i,N';\Z) \to CF^*(N_{i+1},N';\Z) 
$$
are inclusions and the differentials vanish on these chain complexes. 

\begin{figure}
  \begin{center}
    \def\svgwidth{0.3\textwidth}
    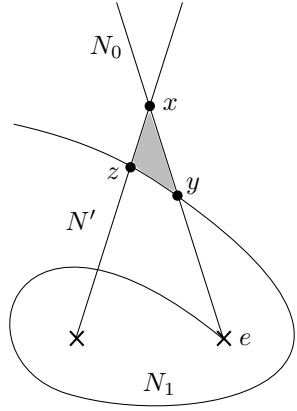
  \end{center}
  \caption{$N_0$ and $N_1$}
  \label{N0N1_fig}
\end{figure}

We also have 
$$
B_\Z\cong HW^*(N_0,N_0;\Z) \cong \lim_{i\to \infty} HF^*(N_i,N_0;\Z).
$$
For $i>0$, the Morse--Bott Floer cochain complex for $CF^*(N_i,N_0;\Z)$ is 
\begin{equation} \label{HF(N1,N0)}
H^*(S^1)\oplus \bigoplus_{k=1}^i H^{*+2k}(T^2;\Z)
\end{equation}
and we have again that the continuation maps are inclusions and the differentials vanish. 

As we saw after Proposition \ref{HW(N)}, the free abelian group $HW^*(N,N;\Z)$ has two generators in degree $-2$, denoted by $u e_1$ and $u e_2$. 

\begin{remark}
At several points in this paper, including the proof of the next result, we will explicitly compute certain products $\mu^2$ via counts of holomorphic curves. Since we will always be in a position where we can compute the product on cohomology, and since the relevant holomorphic curves will always project to triangles in $\C$ over which the Lefschetz fibrations of interest are trivial, it will suffice to make all the calculations using a product complex structure, for which the relevant holomorphic curves are regular. 
\end{remark}

\begin{lemma} \phantomsection \label{ue geometric}
\begin{enumerate}
 \item The fundamental class of $S^1$ in \eqref{HF(N1,N0)}, with $i=1$, represents the class $\pm e = \pm(e_1 + e_2)\in HW^0(N_1,N_0,\Z)$, where $e$ is the unit. 
 \item The fundamental class of $T^2$ in \eqref{HF(N1,N0)}, with $i=1$, represents the class $\pm u e_1 \pm u e_2 \in HW^{-2}(N_1,N_0;\Z)$. \label{part: part 2}
\end{enumerate}
\end{lemma}

\begin{proof}
The statement in (1) follows from the fact that the canonical map $H^*(S^1) \to HW^*(N_0,N_0;\Z)$ is a ring map, so it preserves units. 

For (2), it is convenient to also consider the Lagrangian $N'$. The product $\mu^2$ gives a map 
$$
HF^0(N_0,N';\Z) \otimes HF^{-2}(N_1,N_0,\Z) \to HF^{-2}(N_1,N';\Z).
$$
Figure \ref{N0N1_fig} will be useful to understand the map 
\begin{equation} \label{0 to -2}
HF^0(N_0,N';\Z) \to HF^{-2}(N_1,N';\Z)
\end{equation}
given by right multiplication with the fundamental class of $T^2$ in $HF^{-2}(N_1,N_0,\Z)$, which lies over the intersection point $y$ in Figure \ref{N0N1_fig}. Note that $HF^0(N_0,N';\Z)\cong HW^0(N_0,N';\Z) \cong \Z^2$ is generated by classes that lie over the point $x$ in the Figure, and that $HF^{-2}(N_1,N';\Z)\cong HW^{-2}(N_1,N';\Z) \cong \Z^2$ is generated by classes that lie over the point $z$. The product can now be computed by lifting the shaded triangle. Since the fibration is trivial over this triangle, there is a $T^2$-family of such lifts. Inserting the fundamental class of $T^2$ over $y$ does not impose any constraint on this family of disks, which implies that the map \eqref{0 to -2} is an isomorphism. Since we are working over $\Z$, this means that the fundamental class of $T^2$ represents $\pm u e_1 \pm u e_2$, as wanted. 
\end{proof}

\begin{remark}
We are not specifying if the two signs in part \eqref{part: part 2} above are the same, since that is not necessary where this result is applied later in the paper. Nevertheless, a surgery argument related to Lemma \ref{Ltau is cone} below should imply that the signs are equal, hence $\pm u e_1 \pm u e_2 = \pm u e$. 
\end{remark}

\subsection{Computations in $T^*S^n$}

We begin by assuming that {\bf $n$ is odd}. The wrapped Fukaya category $\W^{\Z/2\Z}(T^*S^n;\K_R)$ contains objects of the form $(S^n,\alpha[pt])$, were $\alpha \in \K_{R,0}$ and $[pt] \in H^n(S^n;\K_R)$ is the class of a point. 
We want to understand how a cotangent fiber $F$ acts on such an object. 

Let $F_i$ and $F'$ be as in the previous section. Given $a \in HF^*(F_i,F';\K_R)$ and $X\in \W^{\Z/2\Z}(T^*S^n;\K_R)$, define a map 
\begin{align*}
 \psi_a^X \colon HF^*(F',X;\K_R) & \to HF^{*+\deg(a)}(F_i,X;\K_R) \\
 x & \mapsto \mu^2(x,a)
 \end{align*}

\begin{figure}
  \begin{center}
    \def\svgwidth{0.4\textwidth}
    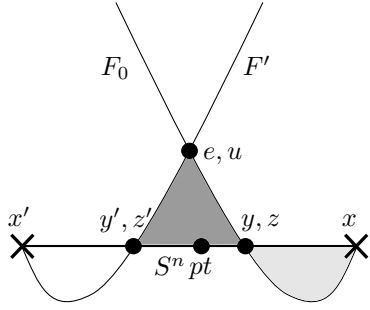
  \end{center}
  \caption{The chain complexes $CF^*(F_0,(S^n, \alpha[pt]))$ and $CF^*(F',(S^n, \alpha[pt]))$}
  \label{HF(F,Sn)_fig}
\end{figure}

\begin{lemma} \phantomsection \label{HF(F,Sn) odd}
\begin{enumerate}
\item There is an isomorphism \label{HF odd}
$$
HF^*(F,(S^n, \alpha[pt]);\K_R) \cong \K_R.
$$
\item  \label{u act on Sn odd}  Using the identification in Lemma \ref{u in F} of $e\in HF^{0}(F_0,F';\Z)$ with the class of a point in $S^{n-1}$, and of $u\in HF^{1-n}(F_0,F';\Z)$ with the fundamental class of $S^{n-1}$, we have  
$$\psi_{u}^{(S^n,\alpha[pt])} = \pm\alpha \, \psi_{e}^{(S^n,\alpha[pt])}.$$
\end{enumerate}
\end{lemma}

\begin{proof}
As we saw, in the Lefschetz fibration description $\pi_n : X_n\to \C$ of $T^*S^n$ the zero section $S^n$ is the Lagrangian lift of the interval $[-1,1] \subset \C$.

For part \eqref{HF odd}, we can replace a cotangent fiber $F$ with its Hamiltonian-isotopic Lagrangians $F_0$ and $F'$, as in the previous section. Recall that these are lifts of paths out of the critical values that intersect the interval $[-1,1] \subset \C$ transversely and only at one of the endpoints of the interval. Then, $CF^*(F_0,(S^n, \alpha[pt]);\K_R)$ has a single generator in degree 0, and the result follows. The same is true replacing $F_0$ with $F'$. 

Let us give an alternative argument, with an eye towards part \eqref{u act on Sn odd}. This time, let $F_0$ and $F'$ be lifts of paths that intersect the interior of $[-1,1]$, as in Figure \ref{HF(F,Sn)_fig}. We start with $F_0$. The chain complex $CF^*(F_0,(S^n, \alpha[pt]);\K_R)$ now has generators $x,y,z$ in degrees $-n$, $1-n$ and 0, respectively ($y$ is the maximum and $z$ the minimum of an auxiliary Morse function on the component of $S^n \cap F_0$ that is diffeomorphic to $S^{n-1}$), see Figure \ref{HF(F,Sn)_fig}. The fact that $\partial x$ is of the form $\pm T^A y$, where $A$ is the $\sigma$-area of the lightly shaded bigon (recall the definition of $\sigma$ in Section \ref{S:construct Lagrangians}), follows from the fact that the algebraic count of lifts of the shaded strip is $\pm 1$. That can be seen using the invariance under compactly supported Hamiltonian isotopies in $\C^n$ of $HF^*(\R^n,i\R^n)$, which is of rank 1. It follows that the cohomology is of rank 1, generated by $z$. There is a similar argument for $F'$ instead of $F_0$, with $z'$ now being the maximum of an auxiliary Morse function on $S^{n-1}$.  

To prove \eqref{u act on Sn odd}, we use again the representation of $F_0$ and $F'$ in Figure \ref{HF(F,Sn)_fig}. The dark triangle in Figure \ref{HF(F,Sn)_fig} does not contain critical values of $\pi_n$, so the restriction of the 
Lefschetz fibration to that triangle is trivial. 
The triangle hence lifts to an $S^{n-1}$-family 
of holomorphic triangles with the appropriate Lagrangian boundary conditions. This family can be made rigid by using $e\in CF^*(F_0,F';\K_R)$ (represented by a minimum) as an input in  
$$
\psi_{e}^{(S^n,\alpha[pt])}(z') = \mu^2(z',e)= \pm T^B z,
$$
where $B$ is the $\sigma$-area of the dark triangle.  
Similarly, the family of lifted triangles can be rigidified by using the bounding cochain $\alpha [pt]$ as an input in 
$$
\psi_{u}^{(S^n,\alpha[pt])}(z') = \mu^3(\alpha [pt],z',u)=\pm T^B \alpha z.
$$
The result now follows.
\end{proof}

Consider now the case of {\bf even $n$}. Recall that we equip $S^n$ with the trivial rank 2 vector bundle of mixed degree $E=\K_R\oplus \K_R[1]$, and with bounding cochains of the form $b_{\alpha,\beta} = \begin{pmatrix}
             0 & \beta \\
             \alpha & 0 
            \end{pmatrix}_{[pt]}$, 
such that $\alpha,\beta \in \K_{R,0}$ and $[pt]\in H^n(S^n;\Z)$ is represented by the minimum of a perfect Morse function on $S^n$. Let $F_0, F'$ be as before. 

\begin{lemma} \phantomsection \label{HF(F,Sn) even}
\begin{enumerate}
\item There is an isomorphism \label{HF even}
$$
HF^*(F,(S^n, b_{\alpha,\beta});\K_R) \cong \K_R\oplus \K_R[1]. 
$$
\item \label{u act on Sn even}  Using the identification in Lemma \ref{u in F} of $e\in HF^{0}(F_0,F';\Z)$ with the class of a point in $S^{n-1}$, and of $u\in HF^{1-n}(F_0,F';\Z)$ with the fundamental class of $S^{n-1}$, we have  
$$\psi_{u e}^{(S^n,b_{\alpha,\beta})} = \pm \begin{pmatrix}
             0 & \beta \\
             \alpha & 0 
            \end{pmatrix} \, \psi_{e}^{(S^n,b_{\alpha,\beta})}.$$
\end{enumerate}
\end{lemma}
\begin{proof}
The proof of \eqref{HF even} is similar to Lemma \ref{HF(F,Sn) odd}. One can again replace $F$ with either $F_0$ or $F'$ as in Figure \ref{HF(F,Sn)_fig}. We obtain a $\K_R$-basis $v_0,v_1$ for $HF^*(F_0,(S^n, b_{\alpha,\beta});\K_R)$, where $v_0,v_1$ is the standard basis for the fiber of $E=\K_R\oplus \K_R[1]$ at $z$ (the fiber minimum) indicated in Figure \ref{HF(F,Sn)_fig}. Similarly, we denote by $v'_0, v'_1$ the analogous basis for $HF^*(F',(S^n, b_{\alpha,\beta});\K_R)$, with $z$ replaced by $z'$ (the fiber maximum) in Figure \ref{HF(F,Sn)_fig}.

The result in \eqref{u act on Sn even} follows again from the study of lifts of the dark triangle in Figure \ref{HF(F,Sn)_fig}. Once more, the lifts of the triangle can be rigidified either by taking $e$ as an input in $\mu^2$ or by inputing the bounding cochain $b_{\alpha,\beta}$ in $\mu^3$. Taking bases $v_i$ and $v'_i$ above, we get
$$
\psi_{e}^{(S^n,b_{\alpha,\beta})}(v_i') = \mu^2(v_i',e)= \pm T^B v_i,
$$
for $i = 0, 1$, where $B$ is the $\sigma$-area of the dark triangle. We also get 
$$
\psi_{u}^{(S^n,b_{\alpha,\beta})}(v_0') = \mu^3(b_\alpha,v_0',u)=\pm T^B \alpha v_1
$$
and 
$$
\psi_{u}^{(S^n,b_{\alpha,\beta})}(v_1') = \mu^3(b_{\alpha,\beta},v_1',u)=\pm T^B \beta v_0,
$$
as wanted.
\end{proof}

\begin{figure}
  \begin{center}
    \def\svgwidth{0.4\textwidth}
    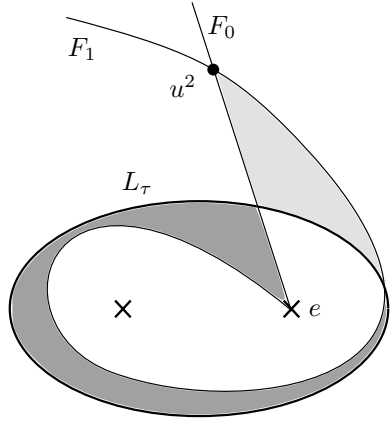
  \end{center}
  \caption{$L_\tau$ as a cone on morphisms between $F_0$ and $F_1$}
  \label{cone_fig}
\end{figure}

We now consider the Lagrangians $L_\tau$, which are diffeomorphic to $S^1\times S^{n-1}$. Let $U\in U_{\K_R}^*$ be an invertible unitary element in the Novikov field and take 
$\alpha := T^{-2(n-1)\tau} U^{-1}\in \K_R \setminus \K_{R,0}$. 
If $n>2$, write $L_\alpha$ for the Lagrangian $L_\tau$ equipped with the unitary local system $\xi$ in the trivial rank 1 $\K_R$-bundle over $L_\tau$, such that $U$ is the holonomy of $\xi$ along a loop that projects in degree 1 to the curve $C_\tau$ (recall that we think of $C_\tau$ as the curve $C$ in Figure \ref{F_w_fig}). 
If $n=2$, recall that we picked a basis $h_1,h_2$ for $H_1(L_\tau;\Z)$ in Lemma \ref{L:disk potential L_C}, to write the disk potential of $L_\tau$. The curve $h_1$ projects in degree 1 to $C_\tau$ and $h_2$ is a fiber of $\pi|_{L_\tau}$. In Remark \ref{R:crit pts of potentials}, we observed that the Floer cohomology of $(L_\tau,\xi)$ is non-trivial precisely when $\xi$ is a local system with holonomy $-1$ around $h_2$. Write $L_\alpha$ for $(L_\tau,\xi)$, such that the holonomy of $\xi$ is $U$ around $h_1$ and $-1$ around $h_2$.

If the $\sigma$-areas of the two shaded regions in Figure \ref{cone_fig} are the same, then the figure suggests that $L_\tau$ should be equivalent to surgery on morphisms supported on the two connected components of the intersection $F_1\cap F_0 = \{*\}\cup S^{n-1}$. Recall that surgery on an intersection point of two Lagrangians corresponds in the Fukaya category to taking the cone on the morphism given by the intersection point, see Chapter 10 of \cite{FOOO}. This motivates the following result.

\begin{lemma} \label{Ltau is cone}
For the appropriate choice of spin structure, $L_\alpha$ is isomorphic in $\mathcal W_{\mon}^{\Z/2\Z}(T^*S^n;\K_R)$ to $\Cone( u^2 - \alpha e)$, where $ u^2 - \alpha e$ is thought of as a morphism in $HW^{\rm even}(F,F)$.  In particular,  
$$
HF^*(F, L_\alpha;\K_R) \cong H^*(S^{n-1};\K_R),
$$ 
as $\Z/2\Z$-graded free $\K_R$-modules.
\end{lemma}
\begin{figure}
  \begin{center}
    \def\svgwidth{0.8\textwidth}
    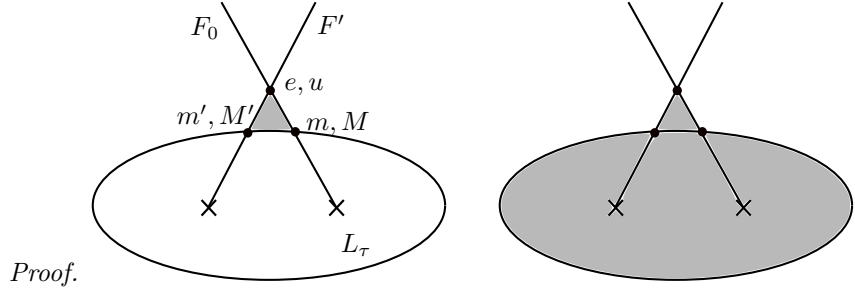
  \end{center}
  \caption{The action of $u$ on $L_\alpha$}
  \label{cone1_fig}
\end{figure}
\begin{proof}
Given a monic polynomial $p(u) = u^d  + a_{d-1} u^{d-1}+ \ldots +a_0$ in $\K_R[u] \cong HW^*(F,F;\K_R)$, the object $\Cone(p(u))$ is such that $HF^*(F,\Cone(p(u));\K_R)$ is a free $\K_R$-module of rank $d$. The right action of $u$ on $HF^*(F,\Cone(p(u));\K_R)$ is by the transpose of the companion matrix to $p(u)$:
\begin{equation*} 
\begin{pmatrix}
0 & 1 & 0 & \cdots & 0 \\
0 & 0 & 1 & \cdots & 0\\
\vdots & \vdots & \vdots & \ddots & \vdots \\
-a_0 & -a_1 & -a_2 & \cdots & -a_{d-1} 
\end{pmatrix}.
\end{equation*}
Hence, we want to show that $HF^*(F,L_\alpha;\K_R)$ is a free $\K_R$-module of rank 2, where $u$ acts on the right as 
\begin{equation} \label{eq: u act on cone1}
\begin{pmatrix} 
0 & 1 \\
\alpha & 0
\end{pmatrix}
\end{equation}

Let us represent the Hamiltonian isotopy class of $F$ by $F_0$ and by $F'$, as before. In Figure \ref{cone1_fig}, we see that $F_0 \cap L_\tau $ is diffeomorphic to $S^{n-1}$. 
Choosing an auxiliary perfect Morse function on this sphere, we get a chain model for $CF^*(F_0,L_\alpha;\K_R)$ whose generators are the minimum $m$ and the maximum $M$. We can similarly get generators $m', M'$ for $CF^*(F',L_\alpha;\K_R)$. We will first work with coefficients in $\K_\Z$, and then argue that the case of $\K_\C$-coefficients follows. In particular, we will begin by assuming that $U\in U_{\mathbb K_\Z}^*$ and then consider the more general case of $U\in U_{\mathbb K_\C}^* = U_{\mathbb K}^*$.

Recall Lemma \ref{u in F}. The element $e\in CF^0(F_0,F';\K_\Z)$ (the minimum in its $S^{n-1}$ fiber) acts on $CF^*(F',L_\alpha,\K_\Z)$ by 
$$
\psi_{e}^{L_\alpha}(M') = \mu^2(M',e)= \pm T^A m,
$$
by taking lifts of the shaded triangle on the left in Figure \ref{cone1_fig}. We denote by $A$ the $\sigma$-area of this triangle. Note that the Lefschetz fibration is trivial over the triangle, so it has an $S^{n-1}$-family of holomorphic lifts. The remaining contributions to right multiplication by $e$ must come from lifts of the shaded triangle on the right in Figure \ref{cone1_fig}.  Since $e$ represents a cohomological unit in $HW^*(F,F;\K_\Z)$, it acts by an isomorphism over $\K_\Z$ and we conclude that the lifts of that triangle contribute to   
$$
\psi_{e}^{L_\alpha}(m') = \mu^2(m',e)= \pm T^B U M, 
$$
where $B$ is the $\sigma$-area of that right triangle in the plane. Note that these lifted triangles pick up holonomy $U$. 
The same holomorphic triangles determine the action of $e$ over $\K=\K_\C$. We conclude that $\psi_{e}^{L_\alpha}$ is given by the same formulas over $\K$ as over $\K_\Z$, even when we take $U\in U_{\mathbb K}^*$. 

The element $u\in CF^{1-n}(F_0,F';\K_R)$ is represented by the maximum in the same $S^{n-1}$ fiber as $e$, and acts on $CF^*(F',L_\alpha;\K_R)$ by 
$$
\psi_{u}^{L_\alpha}(M') = \mu^2(M',u)= \pm T^A M,
$$
and
$$
\psi_{u}^{L_\alpha}(m') = \mu^2(m',u)= \pm T^A m.
$$
In both cases, this corresponds to lifting the triangle on the left in Figure \ref{cone1_fig}.

Observe that $B = A + 2(n-1)\tau$, so we can write 
\begin{equation} \label{matrix alpha}
\psi_u^{L_\alpha} = \begin{pmatrix} 
0 & 1 \\
\pm \alpha & 0
\end{pmatrix} \psi_e^{L_\alpha}. 
\end{equation}
To get a positive sign in $\alpha$ as in \eqref{eq: u act on cone1}, we note that by changing the spin structure on $L_\tau$ we can change the sign of $\psi_{e}^{L_\alpha}(m')$ (which comes from lifting the shaded triangle on the right in Figure \ref{cone1_fig}). This has the desired has the effect of replacing $\alpha$ by $-\alpha$ in the matrix in \eqref{matrix alpha}. 

We still need to show that $HF^*(F,L_\alpha;\K_R)\neq 0$. We prove that $HF^*(F',L_\alpha;\K_R)\neq 0$. If $n$ is odd, then this is obvious, since the indices of the generators $m',M'$ have the same parity, so the differential is zero. Observe that the case $n=2$ is addressed in Remark \ref{R:crit pts of potentials}. For general even $n$, we write 
$$
\mu^1(m) = \kappa_1 M, \qquad \mu^1(M) = \kappa_2 m, \qquad \mu^1(m') = \kappa_1' M', \qquad \mu^1(M') = \kappa_2' m',
$$
for some $\kappa_1, \kappa_2, \kappa_1', \kappa_2' \in \K$. The Leibniz rule (and the fact that $\mu^1(u)=0$) yields
\begin{align*}
\mu^1(\mu^2(m', u)) &= \mu^2(\mu^1(m'), u) = \kappa_1' \mu^2(M', u) = \pm \kappa_1' T^A M \\
&= \pm \mu^1 (T^A m) = \pm T^A \kappa_1 M \Longrightarrow \kappa_1 = \pm \kappa_1'
\end{align*}
and 
\begin{align*}
\mu^1(\mu^2(m', e)) &= \mu^2(\mu^1(m'), e) = \kappa_1' \mu^2(M', e) = \pm \kappa_1' T^A m \\
&= \pm \mu^1 (T^B U M) = \pm T^B U \kappa_2 m \Longrightarrow T^{B-A} U \kappa_2 = \pm \kappa_1' = \pm \kappa_1.
\end{align*}
Therefore, 
$$
\mu^1\circ \mu^1(M) = \kappa_2 \mu^1(m) = \pm T^{B-A} U \kappa_2^2 M.
$$
But $\mu^1\circ \mu^1=0$, because $F$ and $L_\alpha$ both have vanishing disk potential. Since $T^{B-A} U\neq 0$, we conclude that $\kappa_2 = \kappa_1=\kappa_1'=0$. A similar argument shows that $\kappa_2'=0$, and implies that $HF^*(F',L_\alpha;\K_R)\neq 0$, as wanted.
\end{proof}

As we saw, Lemma \ref{Ltau is cone} can be rephrased as saying that $HF^*(F,L_\alpha;\K_R)$ is isomorphic to $\K_R^2$, if $n$ is odd, and to $\K_R\oplus \K_R[1]$, if $n$ is even, and that the action of $u$ is represented by the matrix \eqref{eq: u act on cone1}.
To relate this with the generation results for modules that will be discussed below, it is convenient to restrict our attention to $\K=\K_\C$, which is an algebraically closed field. Since the eigenvalues of the matrix \eqref{eq: u act on cone1} are  $\pm \sqrt{\alpha}$ (the two square roots of $\alpha$ in $\K$), we conclude the following.

\begin{corollary} \label{u act on Ltau}
If $n$ is odd, then $HF^*(F',L_\alpha;\K)$ and $HF^*(F_0,L_\alpha;\K)$ have bases in which $\psi_{u e}^{L_\alpha} = \begin{pmatrix}
 \sqrt{\alpha} & 0 \\
 0 & - \sqrt{\alpha}
 \end{pmatrix} \, \psi_{e}^{L_\alpha}$.
\end{corollary}

We are now ready to prove the following result, up to Corollary \ref{S generate} below. 

\begin{theorem} \label{split generators Sn}
The category $\mmod_{pr}(A_{\K})$ is split-generated by the collection of right $A_{\K}$-modules 
\begin{itemize}
\item $
 \{HF^*(F,(S^n,\alpha[pt]);\K)\}_{0\leq \val(\alpha) \leq \infty} \cup \{HF^*(F, L_\alpha;\K)\}_{\val(\alpha) < 0}
 $, if $n$ is odd;
\item $
\{HF^*(F,S^n;\K)\}\cup  \{HF^*(F,(S^n,b_{\alpha,1});\K)\}_{0\leq \val(\alpha) <\infty} \cup \{HF^*(F, L_\alpha;\K)\}_{\val(\alpha) < 0}
 $, if $n$ is even.
\end{itemize}
\end{theorem}

\begin{proof}
In the $n$ odd case, if $\val(\alpha) \geq 0$, then Lemma \ref{HF(F,Sn) odd} 
implies that  
$$
HF^*(F_0,(S^n,\alpha[pt]);\K) \cong S_{\pm\alpha}
$$
as right $A_{\K}$-modules, where $S_\alpha$ is the 1-dimensional (over $\K$) right $A_{\K}$-module on which $u\in A_{\K}$ acts as multiplication by $\alpha$ (as in Lemma \ref{L:triangulated closure} below). 

If $\val(\alpha) < 0$, Lemma \ref{Ltau is cone} and Corollary \ref{u act on Ltau} imply that
$$
HF^*(F, L_\alpha;\K) \cong S_{\sqrt{\alpha}} \oplus S_{-\sqrt{\alpha}} 
$$
as right $A_{\K}$-modules. 

Corollary \ref{S generate} below now implies the result when $n$ is odd. The case of $n$ even is analogous, where this time we apply Lemma \ref{HF(F,Sn) even} instead of Lemma \ref{HF(F,Sn) odd} and Corollary \ref{S tilde generate} instead of Corollary \ref{S generate}.  
\end{proof}

The following is a version of Theorem \ref{generate F} from the Introduction. 

\begin{corollary} \label{generate Fuk Sn}
The category $\F^{\Z/2\Z}_{mon}(T^*S^n;\K)$ is split-generated by the collection of objects 
\begin{itemize}
\item $\{(S^n,\alpha[pt])\}_{0\leq \val(\alpha) \leq \infty} \cup \{L_\alpha\}_{\val(\alpha) < 0}$, when $n$ is odd;
\item $\{S^n\}\cup \{(S^n,b_{\alpha,1})\}_{0\leq \val(\alpha) <\infty} \cup \{L_\alpha\}_{\val(\alpha) < 0}$, when $n$ is even.
\end{itemize} 
\end{corollary}
\begin{proof}
This follows from Theorem \ref{split generators Sn} and Proposition \ref{C:Yoneda ff}. 
\end{proof}

\subsection{Computations in $T^*S^3$}

We now want to study how $u\in HW^*(F,F;\K_R)\cong \K_R[u]= A_{\K_R}$ acts on the tori $T^3_\tau$. Recall from Lemma \ref{L:disk potential T_C} that the disk potential of $T^3_\tau$ can be computed in a basis $h_1,h_2,h_3$ of $H_1(T^3_\tau;\Z)$, where $h_1$ is a loop projecting bijectively to the curve $C_\tau\subset \C\setminus \{\pm 1\}$ (that $T_\tau^3$ covers), while $h_2$ and $h_3$ are vanishing circles that project to points under the fibration. As observed in Remark \ref{R:crit pts of potentials}, the critical points of the disk potential that belong to $(U_{\K_R}^*)^3$ correspond to unitary local systems on $T^3_\tau$, whose holonomy around $h_1$ is arbitrary, and whose holonomy around each of $h_2$ and $h_3$ is $-1$.   

Given $U\in U_{\K_R}^*$, let $\alpha := T^{-2\tau} U^{-1} \in \K_R\setminus \K_{R,0}$ and denote by $T_\alpha$ the Lagrangian $T^3_\tau$ equipped with a unitary local system $\xi$ in the trivial rank 1 $\K_R$-bundle, whose holonomy around $h_1$ is $U$, and whose holonomies around $h_2$ and $h_3$ are $-1$. 

Given $a \in HF^*(N_1,N_0;\K_R)$ and $X\in \W^{\Z/2\Z}_{\mon}(T^*S^3;\K_R)$, define a map 
\begin{align*}
 \phi_a^X \colon HF^*(N_0,L;\K_R) & \to HF^*(N_1,L;\K_R) \\
 x & \mapsto \mu^2(x,a)
\end{align*}

\begin{figure}
  \begin{center}
    \def\svgwidth{0.5\textwidth}
    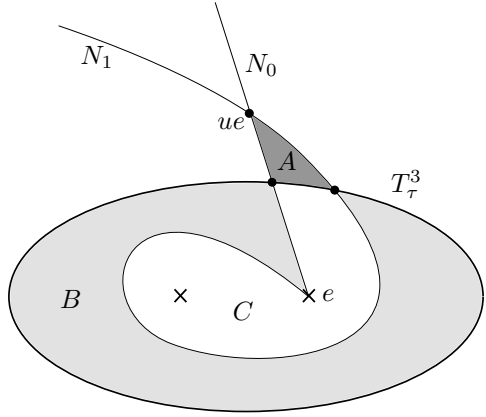
  \end{center}
  \caption{The action of $u$ on $T^3_\alpha$}
  \label{HF(N,L)_fig}
\end{figure}

\begin{lemma} \phantomsection \label{HF(N,T)}
For the appropriate choice of spin structure on $T_\alpha$,
\begin{enumerate}
\item there is an isomorphism
$$
HF^*(N,T_\alpha;\K) \cong H^*(T^2;\K),
$$ 
possibly with a degree shift;
\item using the identification in Lemma \ref{ue geometric} of $\pm e\in HF^{0}(N_1,N_0;\Z)$ with the fundamental class of $S^1$, and of $\pm u e_1 \pm u e_2\in HF^{-2}(N_1,N_0;\Z)$ with the fundamental class of $T^2$, we can pick bases so that 
$$
\phi_{\pm u e_1 \pm u e_2}^{T_\alpha} = 
\begin{pmatrix}
\pm \alpha & & & \\
0 & \pm \alpha & & \\
0 & 0 & \pm \alpha & \\
* & 0 & 0 & \pm \alpha 
\end{pmatrix}
 \, \phi_{e}^{T_\alpha}
$$
where the $4\times 4$ matrix above is lower triangular. 
\end{enumerate}
\end{lemma}
\begin{proof}
We begin with (1). By Remark \ref{R:crit pts of potentials}, $T_\alpha$ corresponds to a critical point of the disk potential of the torus $T_\tau^3$. Hence, $HF^*(T_\alpha,T_\alpha;\K)\cong H^*(T^3;\K)$ has rank 8. Observe that for the Lagrangian lifts $N_i$ of the paths $\eta_i$ in Figure \ref{F_w_fig}, each of the graded $\K$-vector spaces $CF^*(N_i,T_\alpha;\K)$ is isomorphic to $H^*(T^2;\K)$. This has rank 4, so the rank of $HF^*(N_i,T_\alpha;\K)$ can only be 0, 2 or 4. 

On the other hand, by Proposition \ref{HW(N)}, 
$$HF^*(N_i,T_\alpha,\K) \cong HF^*(F,T_\alpha;\K)\oplus HF^*(F,T_\alpha;\K)[1].$$
Since $F$ split-generates the monotone wrapped Fukaya category and $T_\alpha$ is a non-trivial object, we conclude that the rank of $HF^*(F,T_\alpha;\K)$ is 1 or 2. Denote by $M$ this $\K[u]$-module. The fully faithfulness of the Yoneda embedding implies that 
$$
HF^*(T_\alpha, T_\alpha;\K) \cong \Ext_{\K[u]}^*(M,M)
$$
and the rank of the right side cannot be 8 if $M$ has rank 1, since the space of endomorphisms of a skyscraper sheaf in the derived category of a smooth curve has rank 2  (see Lemma \ref{lemma: ExtSaSa} below).  We conclude that $M$ has rank 2.%
\footnote{The referee suggested the following alternative argument for calculating that the rank of $M=HF^*(F,T_\alpha;\K)$ is 2, avoiding the computation of $\Ext$. If one thinks of the Lagrangian $T_\tau^3$ as coming from surgery on two cleanly intersecting Lagrangian spheres, as explained in Remark \ref{rmk:flux}, then it is clear that $T_\tau^3$ intersects a cotangent fiber $T_p^*S^3$ transversely in two points, for all $p$ contained in a non-empty open subset of $S^3$. Since $F$ is Hamiltonian-isotopic to a cotangent fiber, this implies that the rank of $M$ is 0 or 2. But since $T_\alpha$ has non-trivial Floer cohomology, the rank of $M$ cannot vanish.}
Hence,  $HF^*(N_i, T_\alpha;\K)$ has rank 4 and the differential on $CF^*(N_i, T_\alpha;\K)$ vanishes. This implies the statement in (1). 

To prove (2), we use Figure \ref{HF(N,L)_fig}. First, we need to determine the image of $\pm e$ and $\pm u e_1 \pm u e_2$ under the functor $\mathcal G_1 : \mathcal W^\Z(T^*S^3;\Z) \to \mathcal W^\Z(T^*S^3;\K_R)$ defined in Section \ref{S:W}. By Lemma \ref{ue geometric}, if we pick a perfect Morse function on $N_0\cap N_1\cong S^1\cup T^2$, we know that $\pm e$ is given by the maximum (which we denote by $x$) on the component $S^1$, and $\pm u e_1 \pm u e_2$ is given by the maximum (which we denote by $y$) on the component $T^2$. Pick primitives $f_i$ for the restriction of the Liouville form $\lambda$ on $T^*S^3$ to the $N_i$, $i\in \{0,1\}$. Assume that $f_0$ and $f_1$ both vanish at $x$. Then, $\mathcal G_1(\pm e) = T^{f_1(x)-f_0(x)} x = x \in CF^0(N_1,N_0;\K_R)$. Similarly, $\mathcal G_1(\pm u e_1 \pm u e_2) = T^{f_1(y)-f_0(y)} y \in CF^{-2}(N_1,N_0;\K_R)$. Under our assumptions, $f_1(y) - f_0(y)$ is the negative of the symplectic area of a strip between $N_1$ and $N_0$, obtained from lifting the union of the darkly shaded and the white triangles in Figure \ref{HF(N,L)_fig}. We denote this area by $A+C$ in the figure. We can conclude that $\pm u e_1 \pm u e_2$ is represented by $T^{-A-C} y$ in $CF^{-2}(N_1,N_0;\K_R)$.

The intersections $N_i \cap T^3_\tau$ are 2-tori, and we can pick perfect Morse functions on these 2-tori to specify bases for $HF^*(N_i,T_\alpha;\K)$. Denote the basis elements $p_0^i, p_1^i, p_2^i, p_3^i$, with Morse indices 0, 1, 1 and 2, respectively. Each of the shaded triangles in the figure lifts to a $T^2$-family of holomorphic triangles with suitable boundary conditions. Using the fundamental classes of $S^1$ and $T^2$ as inputs in $\mu^2$ does not constrain the $T^2$-families of holomorphic triangles. Therefore, for a suitable labeling of the generators, the contribution of $p_i^1$ to $\phi_{\pm u e_1 \pm u e_2}^{T_\alpha}(p_i^0)$ is 
\begin{equation} \label{formula with signs}
\pm T^{-A-C} T^A = \pm \alpha T^B U,
\end{equation}
since the sum of the $\sigma$-areas of the lightly shaded and white triangles in the figure is $B+C = 2\tau$. Similarly, the contribution of $p_i^1$ to $\phi_e^{T_\alpha}(p_i^0)$ is $\pm T^B U$. This justifies the diagonal terms in the $4\times 4$ matrix in (2). The possibly non-zero off-diagonal term in that matrix is given by the lifts of a large triangle in the base, given by the union of the white triangle and the two shaded triangles. Such lifts come in 4-dimensional families, which is why they contribute to an off-diagonal corner in the matrix.   
\end{proof}

\begin{remark}
The fibration is trivial over the darkly shaded triangle in Figure \ref{HF(N,L)_fig}, which is why it has a $T^2$-family of lifts. The fibration is not trivial over the lightly shaded triangle, because one of its vertices is a critical value. Nevertheless, this triangle still has a $T^2$-family of lifts. Note that one could also modify $N_0$ slightly, in a manner similar to what is done in Figure \ref{HF(F,Sn)_fig} for the proof of Lemma \ref{HF(F,Sn) odd}, so that the analogue of the lightly shaded triangle now includes no critical points, even at the corners. 
\end{remark}

We can now prove the following analogue of Theorem \ref{split generators Sn}.
\begin{theorem} \label{split generators S3}
 The collection of $A_{\K}$-modules 
 $$
 \{HF^*(F,(S^3,\alpha[pt]);\K)\}_{0\leq \val(\alpha) \leq \infty} \cup \{HF^*(F,T_\alpha;\K)\}_{\val(\alpha) < 0}
 $$ 
 split-generates the category $\mmod_{pr}(A_{\K})$. 
\end{theorem}

\begin{proof}
We just have to show that we can replace the objects supported on the collection of Lagrangians $\{(S^1\times S^2)_\tau\}_{\tau>0}$ by the objects supported on the collection $\{T^3_\tau\}_{\tau>0}$.

By Proposition \ref{HW(N)}, $N \cong F\oplus F[1]$. In the proof of Lemma \ref{HF(N,T)}, it is shown that the right $A_{\K}$-module $HF^*(F,T_\alpha;\K)$ is of rank 2. It follows from (2) in Lemma \ref{HF(N,T)} that this module must be isomorphic (up to degree shifts) to $S_{\alpha}\oplus S_{\alpha}$, $S_{-\alpha}\oplus S_{-\alpha}$, $S_{\alpha}\oplus S_{-\alpha}$ or $M^2_{\pm\alpha}$ (in the notation of Section \ref{sec: n odd} below). Lemma \ref{lemma: ExtSaSa} below implies that only the first two options (possibly with degree shifts is each of the summands) are compatible with the fact that $HF^*(T_\alpha, T_\alpha;\K) \cong \Ext_{\K[u]}^*(M,M)$ has rank 8. The result now follows from Corollary \ref{S generate}.
\end{proof}

\begin{corollary} \label{generate Fuk}
The category 
$\F^{\Z/2\Z}_{mon}(T^*S^3;\K)$ is split-generated by the collection of objects $\{(S^3,\alpha[pt])\}_{0\leq \val(\alpha) \leq \infty} \cup \{T_\alpha\}_{\val(\alpha) < 0}$.
\end{corollary}
\begin{proof}
This follows from Theorem \ref{split generators S3} and Proposition \ref{C:Yoneda ff}. 
\end{proof}

Now that we understand the $F$-modules associated to the Lagrangians $(S^1\times S^2)_\tau$ and $T^3_\tau$ in $T^*S^3$, we can also prove Theorem \ref{T:S1xS2 and T3}.

\begin{proof}[Proof of Theorem \ref{T:S1xS2 and T3}]
We wish to show that, if we fix $\tau, \tau'>0$, then $\tau = \tau'$ iff $(S^1\times S^2)_\tau$ and $T^3_{\tau'}$ can be equipped with local systems such that their Floer cohomology is non-trivial. 
Let $U\in U_\K^*$ and write $\alpha = T^{-4 \tau} U^{-1}$. Recall that the minimal Maslov number of $(S^1\times S^2)_\tau$ is 4, and that $(S^1\times S^2)_\alpha$ denotes $(S^1\times S^2)_\tau$ equipped with a rank 1 unitary local system with holonomy specified by $U$. The proof of Theorem \ref{split generators Sn} implies that $HF^*(F,(S^1\times S^2)_\alpha;\K) \cong S_{\sqrt{\alpha}} \oplus S_{-\sqrt{\alpha}}$, where $\sqrt{\alpha} = T^{-2\tau} (\sqrt{U})^{-1}$ for some square root $\sqrt{U}\in \K$ of $U$. Write also $\alpha' = T^{-4 \tau'} U^{-1}$ and $\sqrt{\alpha'} = T^{-2\tau'} (\sqrt{U})^{-1}$. The minimal Maslov number of $T^3_{\tau'}$ is 2, and $T_{\sqrt{\alpha'}}$ denotes $T^3_{\tau'}$ with a rank 1 unitary local system of holonomy specified by $\sqrt U$. The proof of Theorem \ref{split generators S3} implies that $HF^*(F, T_{\sqrt{\alpha'}};\K)$ is isomorphic either to $S_{\sqrt{\alpha'}} \oplus S_{\sqrt{\alpha'}}$ or to $S_{-\sqrt{\alpha'}} \oplus S_{-\sqrt{\alpha'}}$ (possibly with degree shifts in the summands). 

The result now follows from Proposition \ref{C:Yoneda ff} and Lemma \ref{lemma: ExtSaSa} below. 
\end{proof}

\section{Intrinsic formality of algebras and modules}

\label{S:Formality algebra}

Recall that $A_{\K}=HW^*(F,F;\K)$ is isomorphic to the polynomial algebra $\K[u]$, where $\deg(u)=1-n$ and $n\geq 2$. From this point on, we will always work over $\K$, and write $A$ instead of $A_\K$ to make the notation lighter. We want to show that $A$ is {\em intrinsically formal}, and will later prove an analogous result for certain types of $A$-modules. Intrinsic formality of $A$ means that if $\mathcal B$ is any $A_\infty$-algebra such that the algebra $H^*(\mathcal B)$ is isomorphic to $A$, then $\mathcal B$ is quasi-isomorphic to $A$ as $A_\infty$-algebras.

Denoting by $|A|$ the algebra $A$ where we {\em forget the grading}, we can define the Hochschild cohomology 
$HH^r(|A|,|A|)$ as the homology of $CC^r(|A|,|A|) := \Hom_\K(|A|^{\otimes r},|A|)$, for $r\geq 0$, with respect to the Hochschild differential, see for instance \cite{WeibelHA}. 

To keep track of the grading on $A$, one can define 
$$
 CC^{r}(A,A[s]) := \Hom^s_\K(A^{\otimes r},A), 
$$
which consists of graded homomorphisms that increase the degree by $s\in \Z$ (we continue to use the cohomological convention under which $A[s]$ is obtained by {\em subtracting} $s$ from all degrees in $A$), see \cite{SeidelThomas}*{Section 4b}. 
The Hochschild differential preserves $s$, so the $CC^{r}(A,A[s])$ are subcomplexes of $CC^r(|A|,|A|)$. For each $s$ we have a direct sum of chain complexes
$$
CC^*(|A|,|A|) = CC^{*}(A,A[s]) \oplus Q^{*,s}
$$
where $Q^{r,s} \subset CC^*(|A|,|A|)$ consists of those homomorphisms that have no term of degree $s$. 
One can identify $Q^{r,s}$ with the quotient $CC^r(|A|,|A|) / CC^{r}(A,A[s])$.
We can conclude that there are inclusions on cohomology
$$
HH^{r}(A,A[s]) \subset HH^r(|A|,|A|).
$$

\begin{remark}
Since $A$ is supported in infinitely many degrees, none of the inclusions 
$$\bigoplus_{s\in \Z} CC^{r}(A,A[s]) \subset CC^r(|A|,|A|) \subset \prod_{s\in \Z} CC^{r}(A,A[s])$$ 
is the identity. 
\end{remark}

By \cite{WeibelHA}*{Corollary 9.1.5}, 
$HH^*(|A|,|A|) \cong \Ext^*_{|A|^e}(|A|,|A|)$, where $|A|^e = |A|\otimes_\K |A|^{\op}$  
(this is isomorphic to $|A|\otimes_\K |A|$, since $|A|$ is commutative). 
Note that $\Ext_{|A|^e}(|A|,|A|)$ can be computed using any projective resolution of $|A|$ as an $|A|^e$-module. We use the {\em Koszul resolution} 
\begin{equation} \label{Koszul res}
0 \to |A|^e \stackrel{f}{\to} |A|^e \stackrel{g}{\to} |A| \to 0,
\end{equation}
where $f(a(u)\otimes b(u)) = a(u) u \otimes b(u) - a(u) \otimes u b(u)$ and $g(p(u),q(u)) = p(u) q(u)$.

The existence of this 2-step resolution implies that $HH^r(|A|,|A|) = 0$ if $r \geq 2$. 
Since $HH^{r}(A,A[s]) \subset HH^{r}(|A|,|A|)$, this implies that $HH^{r}(A,A[s]) = 0$ for all $s$ and for all $r\geq 2$. 

It is known that 
$A$ is intrinsically formal if $HH^{r}(A,A[2-r]) = 0$ for all $r\geq 3$, see \cite{KadeishviliFormality}*{Corollary 4}, \cite{SeidelHMSquartic}*{Section 3} and \cite{SeidelThomas}*{Theorem 4.7}. We can thus conclude that the $\Z$-graded algebra $A$ is intrinsically formal.  
The vanishing for $r=2$ means that it is also not possible to deform the product structure on $A$. 

The previous argument can be adapted to show that, if we collapse the $\Z$-grading of $A$ to a $\Z/2\Z$-grading, $A$ is still intrinsically formal. More specifically, let $CC^{r}(A,A[\even]) \subset CC^{r}(|A|,|A|)$ be the subcomplex of homomorphisms with no odd degree components, and let $CC^{r}(A,A[\odd]) \subset CC^{r}(|A|,|A|)$ be the subcomplex of homomorphisms with no even degree components. We get a decomposition
$$
HH^{r}(|A|,|A|) \cong HH^{r}(A,A[\even]) \oplus HH^{r}(A,A[\odd]). 
$$
In the $\Z/2\Z$-graded case, intrinsic formality of $A$ follows from the simultaneous vanishing
$$
\begin{cases}
HH^r(A,A[\even]) = 0 \text{, for all } r\geq 3 \text{ even} \\
HH^r(A,A[\odd]) = 0 \text{, for all } r\geq 3 \text{ odd}  
\end{cases}
$$
which is again a consequence of the fact that $HH^r(|A|,|A|) = 0$ if $r \notin \{ 0,1\}$.

We can conclude the following. 

\begin{proposition} \label{prop:A intrinsically formal}
$A=\K[u]$ is intrinsically formal as a $\Z$-graded algebra and as a $\Z/2\Z$-graded algebra. 
\end{proposition}

We now discuss right modules over the graded algebra $A$. 
In a manner similar to the previous discussion, let $|A|$, $|M|$ and $|N|$ be the result of forgetting the $\Z$-gradings of the algebra $A$ and of right $A$-modules $M$ and $N$. Then, $\Hom_\K(|M|,|N|)$ is an $A$-bimodule and its Hochschild cochain complex is, for $r\geq 0$,
\begin{align*}
CC^r(|A|,\Hom_\K(|M|,|N|)) &:= \Hom_\K(|A|^{\otimes r},\Hom_\K(|M|,|N|)) \\
& \cong \Hom_\K(|M|\otimes_\K |A|^{\otimes r},|N|). 
\end{align*}

Remembering the $\Z$-gradings, we can denote as before the homomorphisms of degree $s\in \Z$ by
$$
CC^{r}(A,\Hom_\K(M,N)[s]) := \Hom^s_\K(A^{\otimes r},\Hom_\K(M,N)) \cong \Hom^s_\K(M\otimes_\K A^{\otimes r},N).
$$ 
The Hochschild differential preserves $s$ and we get inclusions on cohomology
$$
HH^{r}(A,\Hom_\K(M,N)[s]) \subset HH^r(|A|,\Hom_\K(|M|,|N|)). 
$$
Using again \cite{WeibelHA}*{Corollary 9.1.5}, we get that 
$HH^*(|A|,\Hom_\K(|M|,|N|)) \cong \Ext^*_{|A|^e}(|A|,\Hom_\K(|M|,|N|))$. The 2-step Koszul resolution \eqref{Koszul res} can be used to show that $HH^r(|A|,\Hom_\K(|M|,|N|))=0$ for $r\geq 2$ and for every $M,N$. Consequently, we get $HH^{r}(A,\Hom_\K(M,N)[s])= 0$ for $r\geq 2$ and for all $s\in \Z$. 

\begin{remark}
It is worth pointing out that $HH^*(|A|,\Hom_\K(|M|,|N|))$ is also isomorphic to $\Ext^*_{|A|/\K}(|M|,|N|)$, the {\em relative} $\Ext$, see \cite{WeibelHA}*{Lemma 9.1.9}.
\end{remark}

Say that a $\Z$-graded right $A$-module $M$ is {\em intrinsically formal} if, for every $\Z$-graded right $A_\infty$-module $\M$ over $A$ such that the $A$-module $H^*\M$ is isomorphic to $M$, we have that $\M$ is quasi-isomorphic to $M$ (as $A_\infty$-modules over $A$). In an analogous manner to the Hochschild cohomology criterion for intrinsic formality of graded algebras discussed earlier, it can be shown that if 
$$HH^{r}(A,\Hom_\K(M,M)[1-r]) = 0$$ 
for all $r\geq 2$, then $M$ is intrinsically formal, see \cite{LadoshkinModules}*{Theorem 3.2}. What we saw above implies that every $\Z$-graded right $A$-module is intrinsically formal. 

The same argument could again be adapted to the case of $\Z/2\Z$-graded modules over $A$ (with the grading of $A$ collapsed to $\Z/2\Z$). If $M$ and $N$ are $\Z/2\Z$-graded right $A$-modules, we can define cohomology groups $HH^{r}(A,\Hom_\K(M,N)[s])$ with $s\in \{0,1\}$. This time, we have a decomposition
$$
HH^r(|A|,\Hom_\K(|M|,|N|)) \cong HH^{r}(A,\Hom_\K(M,N)) \oplus HH^{r}(A,\Hom_\K(M,N)[1]).
$$
The sufficient condition for intrinsic formality of $M$ is now given by the simultaneous vanishing
$$
\begin{cases}
HH^r(A,\Hom_\K(M,N)[1]) = 0 \text{, for all } r\geq 2 \text{ even}  \\
HH^r(A,\Hom_\K(M,N)) = 0 \text{, for all } r\geq 2 \text{ odd}
\end{cases}
$$
and this criterion is again met by the discussion above. We can conclude the following.  

\begin{proposition} \label{modules formal}
All $\Z$-graded and all $\Z/2\Z$-graded right $A$-modules are intrinsically formal. 
\end{proposition}

As in Section \ref{SS:Yoneda}, denote by $\mmod(A)$ a category whose objects are right $A$-modules (we do not mean $A_\infty$-modules). These modules are $\Z$- or $\Z/2\Z$-graded, depending on the context. Given two modules $M, N$, define their morphism space to be 
$$\hom^*_{\mmod(A)}(M,N) = \Ext^*_A(M,N)$$
(instead of usual $A$-module homomorphisms). The following is a consequence of the results of this section. To make the notation more uniform, we denote the $A_\infty$-algebra $\mathcal A_\mathbb K=CW^*(F,F;\mathbb K)$ by $\mathcal A$.

\begin{corollary} \label{C:mod(A) is formal}
Passing to cohomology gives a functor  
\begin{align*}
H:\mmod^{A_\infty}(\A) &\to \mmod(A) \\
\M & \mapsto H^*(\M)
\end{align*}
which is a quasi-equivalence (meaning that it induces an equivalence of categories on cohomology). 
The category $\mmod(A)$ is equivalent to the cohomology category of $\mmod^{A_\infty}(\A)$.
\end{corollary}

\begin{proof}
The fact that morphisms on the cohomology category of $\mmod^{A_\infty}(\A)$ are given by $\Ext$ groups is explained in \cite{SeidelBook}*{Remark 2.15}. There is a composition of quasi-equivalences of dg-categories
$$
\mmod(A) \to \mmod^{A_\infty}(A) \to \mmod^{A_\infty}(\A).
$$
The fact that $\A$ is formal (by Proposition \ref{prop:A intrinsically formal}) implies that the functor on the right is a quasi-equivalence, see \cite{SeidelBook}*{Section 2f}.  
The functor on the left is given by inclusion (thinking of $\mmod(A)$ as a dg-category with trivial differentials), and it is a quasi-equivalence by Proposition \ref{modules formal}.  
The functor $H$ in the statement is a quasi-inverse for this composition.
\end{proof}

\section{Generation of categories of modules}

\label{S:Generation modules}

\begin{definition}
Let $\mmod_{pr}(A)$ be the subcategory of $\mmod(A)$, whose objets are finite dimensional right $A$-modules (the subscript stands for {\em proper}).
\end{definition}

The fact that $\C$ is algebraically closed of characteristic zero implies that $\K$ is also algebraically closed, see \cite{FOOOCompactToricI}*{Appendix A}.%
\footnote{In characteristic $p>0$, the polynomial $x^p - x - T^{-1}$ does not have roots in the Novikov field. See \cite{Kedlaya} for a discussion of the algebraic closure of the power series field in positive characteristic.}
This will enable us to study the category $\mmod_{pr}(A)$ using Jordan normal forms. 

\begin{remark}
One should be able to work over the Novikov field $\mathbb K_R$ for a general commutative ring $R$, by allowing as objects monotone Lagrangians equipped with higher rank local systems. In that setting, one would expect to be able to prove analogues of Corollaries \ref{generate Fuk Sn} and \ref{generate Fuk}, showing that every compact Lagrangian object is split-generated by objects supported on $S^n$ and $(S^1\times S^{n-1})_\tau$ (with the latter being replaceable by $T^3_\tau$, if $n=3$), without appealing to the Jordan normal form.  
\end{remark}

Recall that $A = \K[u]$, where $\deg(u) = 1-n$. Since the monotone Fukaya category is $\Z/2\Z$-graded, we will consider two cases, depending on the parity of $n$. 

\subsection{When $n$ is odd} \label{sec: n odd}
Take an object $M \oplus N$ of $\mmod_{pr}(A)$, where $M$ is in degree 0 and $N$ is in degree 1.
Since $\K$ is algebraically closed, if $M\neq 0$ then it has a splitting 
$$
M \cong \bigoplus_{i = 1}^m M_{\alpha_i}^{k_i}
$$
where $\alpha_i\in \K$, $k_i\in \Z_+$ and $M_{\alpha}^{k}$ is the vector space $\K^k$ with a right action of $u$ by the $k\times k$ transposed Jordan block 
\begin{equation} \label{def:J^T}
(J_\alpha^k)^T = \begin{pmatrix}
 \alpha  \\
  1 & \alpha  \\
  & \ddots & \ddots \\
  & & 1 & \alpha  \\
  & & & 1 & \alpha
\end{pmatrix}
\end{equation}
If the module $N$ is non-trivial, then it also has a splitting 
$$
N \cong \bigoplus_{j = 1}^n M_{\beta_j}^{l_j}[1]
$$
for certain $\beta_j\in \K$ and $l_j\in \Z_+$.

Denote the 1-dimensional module $M_\alpha^1$ by $S_\alpha$. The following result was used above in study of the Lagrangian tori $T^3_\tau\subset T^*S^3$. For the calculation with $S_\alpha$, one can use the resolution
$$
\mathbb K[u] \stackrel{u-\alpha}{\to} \mathbb K[u] \to \mathbb K,
$$
and a similar argument works for $M^2_\alpha$. 

\begin{lemma} \label{lemma: ExtSaSa}
Given $\alpha, \alpha'\in \K$, we have 
\begin{equation*} 
\Ext_{A_{\K}}^*(S_\alpha,S_{\alpha'}) \cong \begin{cases}
\K\oplus \K[1] & \text{ if } \alpha = \alpha' \\
0 & \text{ otherwise } 
\end{cases}
\end{equation*}
and 
\begin{equation*} 
\Ext_{A_{\K}}^*(M_\alpha^2,M_{\alpha'}^2) \cong \begin{cases}
\K^2\oplus \K^2[1] & \text{ if } \alpha = \alpha' \\
0 & \text{ otherwise } 
\end{cases}.
\end{equation*}
\end{lemma}

\begin{lemma} \label{L:triangulated closure}
For every $k\in \Z_+$, $M^k_\alpha$ is in the triangulated closure of $S_\alpha$. 
\end{lemma}
\begin{proof}
Observe that there are $A$-module homomorphisms 
$$
\varphi_\alpha^k \colon M_\alpha^k \to S_\alpha
$$
obtained by projecting onto the last coordinate. We can think of an $A$-module homomorphism as a homomorphism of $A_\infty$-modules, and take its cone. 
Recall that $\Cone(\varphi_\alpha^k)$ is the right $A_\infty$-module over $A$ given by the chain complex
$$
(M_\alpha^{k}[1]\oplus S_\alpha, \mu^1 = \varphi_\alpha^k),
$$
with $\mu^2 = (\mu_{M_\alpha^{k}[1]}^2,\mu^2_{S_\alpha})$ and trivial higher $A_\infty$-maps, see \cite{SeidelBook}*{Section (3e)}. We have that $H^*\Cone(\varphi_\alpha^k) \cong M_\alpha^{k-1}[1]$ and so $\Cone(\varphi_\alpha^k)$ is quasi-isomorphic to $M_\alpha^{k-1}[1]$. 

We can now argue by induction on $k$ to prove the statement in the lemma. In detail, since there is a distinguished triangle 
 $$
 M^k_\alpha \to S_\alpha \to M^{k-1}_\alpha[1] \to M^k_\alpha [1],
 $$
 axiom TR2 for triangulated categories (see for instance 
\cite{WeibelHA}*{Definition 10.2.1}) 
implies that there is also a distinguished triangle
 $$
 S_\alpha \to M^{k-1}_\alpha[1] \to M^k_\alpha [1] \to  S_\alpha[1],
 $$
 and by induction on $k$ we get that $M^k_\alpha$ is in the triangulated closure of $S_\alpha$ for all $k\geq 1$.
\end{proof}

We can now conclude the following.

\begin{corollary} \label{S generate}
The category $\mmod_{pr}(A)$ is generated by the collection of modules $\{S_\alpha\}_{\alpha \in \K}$. 
\end{corollary}

\subsection{When $n$ is even} 
This case is more subtle, because now $u$ acts on $A$-modules as an operator of odd degree. Take again an object $M\oplus N$ in $\mmod_{pr}(A)$, where $M$ is a finite dimensional $\K$-vector space in degree 0 and $N$ is finite dimensional in degree 1. We begin by defining some relevant examples. Given $\alpha \in \K$ and $k\geq 1$, let $\tilde M_{\alpha}^k$ be the right $A$-module consisting of the $\K$-vector space $\K^k\oplus \K^k[1]$, with $u$ acting on the right on row vectors by the block matrix 
$\begin{pmatrix}
  0 & I_k \\
  (J_\alpha^k)^T & 0 
\end{pmatrix}$, 
where $(J_\alpha^k)^T$ is as in \eqref{def:J^T} and $I_k$ is the identity. Denote by $\tilde S_{\alpha}$ the module $\tilde M_\alpha^1$. This consists of the $\K$-vector space $\K\oplus \K[1]$, with a right $u$-action by the matrix
$\begin{pmatrix}
  0 & 1 \\
  \alpha & 0 
\end{pmatrix}$. 

For $\alpha = 0$, we need to consider another type of right $A$-module, where the dimensions of the even and odd summands are different. Given $k\geq 1$, let $\tilde N_{0}^k$ be the right $A$-module consisting of the $\K$-vector space $\K^k\oplus \K^{k-1}[1]$, on which $u$ acts by the following $(2k-1)\times(2k-1)$-matrix, where all the empty blocks are understood to be filled with zeros:  
$$
\left(\begin{array}{@{}c|c|c@{}}
 & & 
\\
\hline
  \begin{matrix}
   &  \\
   & 
  \end{matrix}
  & & I_k \\
\hline
  I_k & &
  \begin{matrix}
   &  \\
   & 
  \end{matrix}
\end{array}\right).
$$
Denote by $S_0$ the module $\tilde N_0^1$. This is the $\K$-vector space $\K$, on which $u$ acts as multiplication by zero. Note that in the previous section (when $n$ was assumed 
odd), $S_0$ was also defined as a 1-dimensional module in even degree with trivial $u$-action.

Let us now go back to the general case of a right $A$-module $M\oplus N$, where $M$ and $N$ are finite dimensional, $M$ is in degree 0 and $N$ is in degree 1. Since we will be interested in split-generation of $\mmod_{pr}(A)$, we can assume that $\dim_\K M = \dim_\K N$, by taking a direct sum of $M$ or $N$ with a $\K$-vector space with trivial $u$-action, if necessary. Since $u$ has odd degree, by picking bases for $M$ and $N$, the $u$-action is given by a matrix of the form 
$\begin{pmatrix}
  0 & R \\
  S & 0 
\end{pmatrix},
$ where $R$ and $S$ are square matrices. Since $u^2$ has even degree, it acts by endomorphisms of both $M$ and $N$. As we saw in the case of $n$ odd, we can pick bases for $M$ and $N$ so that the right action of $u^2$ is represented by the transpose of a Jordan matrix. This means that we can assume that $RS$ and $SR$ consist of finitely many transposed Jordan blocks along the diagonal. 

\begin{lemma}
A suitable choice of bases for $M$ and $N$ induces and isomorphism between $M\oplus N$ and a direct sum of modules of the form $\tilde M_\alpha^k$, $\tilde M_0^k[1]$, $\tilde N_0^k$ or $\tilde N_0^k[1]$. 
\end{lemma}
\begin{proof}
If $\alpha\in \K$ is a non-zero eigenvalue of the action of $u^2$ on $M$, and if $v_1,\ldots,v_k$ is a Jordan basis for a Jordan block associated to this $u^2$-action, then $\Span(v_1,\ldots,v_k,v_1 \cdot u,\ldots,v_k \cdot u)$ is a $u$-invariant subspace of $M\oplus N$, and it is isomorphic to $\tilde M_\alpha^k$ as a right $A$-module. The Jordan blocks associated to $\alpha = 0$ are not invertible matrices, which is why we need to also allow summands of the form $\tilde M_0^k[1]$, $\tilde N_0^k$ or $\tilde N_0^k[1]$. 
\end{proof}

\begin{lemma} \label{L:triangulated closure graded}
For every $k\in \Z_+$ and every $\alpha\in \mathbb K$, $\tilde M^k_\alpha$ is in the triangulated closure of $\tilde S_\alpha$. Also, $\tilde M^k_0$ and $\tilde N^k_0$ are in the triangulated closure of $S_0$.
\end{lemma}
\begin{proof}
The fact that the $\tilde M^k_\alpha$ can be generated by $\tilde S_\alpha$ follows inductively from the short exact sequences
$$
0 \to \tilde M^{k-1}_\alpha \to \tilde M^k_\alpha \to \tilde S_\alpha \to 0,
$$
in a manner analogous to the proof of Lemma \ref{L:triangulated closure}. The short exact sequences
$$
0 \to \tilde N^{k-1}_0 \to \tilde N^k_0\to \tilde S_0 \to 0
$$
can be used to show that the $\tilde N^k_0$ can be generated by $S_0$ and $\tilde S_0$. Finally, the short exact sequence 
$$
0 \to S_0 \to \tilde S_0 \to S_0[1] \to 0
$$
implies that $\tilde S_0$ is in the triangulated closure of $S_0$.
\end{proof}

We can now conclude the following. 

\begin{corollary}  \label{S tilde generate}
The category $\mmod_{pr}(A)$ is split-generated by the collection of modules $\{S_0\}\cup\{\tilde S_{\alpha}\}_{\alpha\in \K\setminus 0}$. 
\end{corollary}

\bibliographystyle{alpha}
\bibliography{biblio}

\end{document}